\documentclass[11pt, reqno]{amsart}

\usepackage{amssymb,latexsym,amsmath,amsfonts}
\usepackage{mathrsfs}
\usepackage{graphicx}

\allowdisplaybreaks

\numberwithin{equation}{section}

\theoremstyle{definition}
\newtheorem{definition}{Definition}[section]

\theoremstyle{remark}
\newtheorem{remark}[definition]{Remark}

\theoremstyle{plain}
\newtheorem{theorem}[definition]{Theorem}
\newtheorem{result}[definition]{Result}
\newtheorem{lemma}[definition]{Lemma}
\newtheorem{proposition}[definition]{Proposition}

\newtheorem{corollary}[definition]{Corollary}
\newtheorem{fact}[definition]{Fact}

\newcommand{\eps}{\varepsilon}

\newcommand{\zt}{\zeta}

\newcommand{\zahl}{\mathbb{Z}}  
\newcommand{\nat}{\mathbb{N}}
\newcommand{\excep}{\mathcal{E}}

\newcommand{\Lamb}{\Lambda}
\newcommand{\rep}{\mathcal{R}}
\newcommand{\path}{\mathscr{P}}
\newcommand{\Pth}{\boldsymbol{\mathcal{Z}}}
\newcommand{\Ipth}{\boldsymbol{\mathcal{W}}}
\newcommand{\Opth}{\boldsymbol{\mathcal{X}}}

\newcommand{\bdy}{\partial}
\newcommand{\dom}{D}
\newcommand{\OM}{\Omega}
\newcommand{\Dsc}{\overline{\Delta}}

\newcommand{\norm}{\mathcal{N}}
\newcommand{\bran}{\mathcal{D}}

\newcommand{\smoo}{\mathcal{C}}
\newcommand{\hol}{\mathcal{O}}

\newcommand{\prjc}{{\sf proj}}
\newcommand{\Ptmap}{\Phi_{\!\boldsymbol{\mathcal{Z}}}}

\newcommand{\er}{{\sf Re}}
\newcommand{\mi}{{\sf Im}}

\newcommand{\bcdot}{\boldsymbol{\cdot}}

\newcommand{\lrarw}{\longrightarrow}

\newcommand{\corr}{\varGamma}
\newcommand{\acorr}{{}^\dagger\!\varGamma}
\newcommand\aGa[1]{{}^\dagger\!\Gamma_{{#1}}}
\newcommand{\eff}{{}^\dagger\!F}
\newcommand\trp[1]{{}^\dagger\!{#1}}
\newcommand{\FS}{\omega_{FS}}
\newcommand{\weakST}{\xrightarrow{\text{weak${}^{\boldsymbol{*}}$}}}
\newcommand\Add[1]{\sum_{{#1}}\nolimits^{\prime}}
\newcommand{\IVar}{\Gamma^\bullet}
\newcommand{\zz}{z^\bullet}
\newcommand\mul[1]{\boldsymbol{{#1}}}
\newcommand\interR[2]{[{#1}\,.\,.\,{#2}]}
\newcommand{\Fml}{\mathscr{F}}
\newcommand{\PerrF}{\mathbb{A}}

\newcommand\shortInt[2]{\int_{\raisebox{1pt}{$\scriptstyle {{#1}}$}}\!{#2}}

\newcommand{\inB}{\mathscr{I\!\!B}}
\newcommand{\spU}{{\sf Spec}_{\sf U}}
\newcommand{\opColl}{\mathfrak{B}}
\newcommand{\filJ}{\mathcal{K}}
\newcommand{\Lam}{\boldsymbol{\Lambda}}

\newcommand{\CC}{\mathbb{C}^2}

\newcommand{\cplx}{\mathbb{C}} 

\newcommand{\rea}{\mathbb{R}}
\newcommand{\pro}{\mathbb{P}^1}
\newcommand{\sro}{X}

\begin{document}

\title[Dynamics of correspondences of $\pro$]{The dynamics of holomorphic correspondences \\
of $\boldsymbol{\pro}$: invariant measures and the normality set}

\author{Gautam Bharali}
\address{Department of Mathematics, Indian Institute of Science, Bangalore 560012, India}
\email{bharali@math.iisc.ernet.in}

\author{Shrihari Sridharan}
\address{Chennai Mathematical Institute, Chennai 603103, India}
\email{shrihari@cmi.ac.in}

\thanks{The first author is supported in part by a UGC Centre for Advanced Study grant}


\begin{abstract} 
 This paper is motivated by Brolin's theorem. The phenomenon we wish to demonstrate is
 as follows: if $F$ is a holomorphic correspondence on $\pro$, then (under certain conditions)
 $F$ admits a measure $\mu_F$ such that, for any point $z$ drawn from
 a ``large'' open subset of $\pro$, $\mu_F$ is the weak${}^{\boldsymbol{*}}$-limit
 of the normalised sums of point masses carried by the pre-images of $z$ under the
 iterates of $F$. Let $\eff$ denote the transpose of $F$. Under the condition
 $d_{top}(F) > d_{top}(\eff)$, where $d_{top}$ denotes the topological degree, the
 above phenomenon was established by Dinh and Sibony. We show that
 the support of this $\mu_F$ is disjoint from the normality set of $F$. There are many interesting
 correspondences on $\pro$ for which $d_{top}(F) \leq d_{top}(\eff)$. Examples are the
 correspondences introduced by Bullett and collaborators. When $d_{top}(F) \leq d_{top}(\eff)$,
 equidistribution {\em cannot} be expected to the full extent of Brolin's theorem. However, we prove
 that when $F$ admits a repeller, equidistribution in the above sense holds true.
\end{abstract}
\maketitle

\section{Introduction}\label{S:intro}

The dynamics studied in this paper owes its origin to a work of Bullett
\cite{bullett:dqc88} and to a series of articles motivated by \cite{bullett:dqc88} --- most
notably \cite{bullettPenrose:mqmmg94, bullettHarvey:mqmKgqs00, bullettPenrose:rlshc01, bullFrei:hcmC-lmHb05}.
The object of study in \cite{bullett:dqc88} is the dynamical system that arises
on iterating a certain relation on $\cplx$. This relation is the zero set of a polynomial $g\in \cplx[z_1,z_2]$
of a certain form such that:
\begin{itemize}
 \item $g(\bcdot,z_2)$ and $g(z_1,\bcdot)$ are generically quadratic; and
 \item if $V_g$ denotes the biprojective completion of $\{g=0\}$ 
 in $\pro\times\pro$, then no irreducible component of $V_g$ is of the form 
 $\{a\}\times\pro$ or $\pro\times\{a\}$, where $a\in \pro$.
\end{itemize}
In \cite{bullettPenrose:rlshc01},
this set-up was extended to polynomials $g\in \cplx[z_1,z_2]$ of arbitrary degree that
induce relations $V_g\subset \pro\times\pro$ with similar properties. Since relations can be
composed, it would be interesting to know whether the iterated composition of such a relation
exhibits equidistribution properties in analogy
to Brolin's Theorem \cite[Theorem 16.1]{brolin:isirf65}.
\smallskip

The reader will be aware of recent results by Dinh and Sibony \cite{dinhSibony:dvtma06}
that, it would seem, should immediately solve the above problem. {\em However, key assumptions in the
theorems of \cite{dinhSibony:dvtma06} fail to hold for many interesting correspondences on $\pro$.}
We shall discuss what this assertion means in the remainder of this section.
\smallskip

On the dynamics of multivalued maps between complex manifolds: results of
perhaps the broadest scope are established in \cite{dinhSibony:dvtma06}. We borrow from
\cite{dinhSibony:dvtma06} the following definition.

\begin{definition}\label{D:holCorr}
Let $X_1$ and $X_2$ be two compact complex manifolds of dimension $k$. We say that $\corr$ is
a {\em holomorphic $k$-chain} in $X_1\times X_2$ if $\corr$ is a formal linear combination
of the form
\begin{equation}\label{E:stdForm}
 \corr \ = \sum_{j=1}^N m_j\Gamma_j,
\end{equation}
where the $m_j$'s are positive integers and $\Gamma_1,\dots,\Gamma_N$ are distinct irreducible
complex subvarieties of $X_1\times X_2$ of pure dimension $k$. Let $\pi_i$ denote the projection onto 
$X_i, \ i=1,2$. We say that the holomorphic
$k$-chain {\em $\corr$ determines a meromorphic correspondence of $X_1$ onto $X_2$} if, for each 
$\Gamma_j$ in \eqref{E:stdForm}, $\left.\pi_1\right|_{\Gamma_j}$ and $\left.\pi_2\right|_{\Gamma_j}$
are surjective. $\corr$ determines a set-valued map, which we denote by $F_\corr$, as follows: 
\[
 X_1\ni x \longmapsto \bigcup_{j=1}^N \pi_2\left(\pi_1^{-1}\{x\}\cap \Gamma_j\right).
\]
We call $F_\corr$ a {\em holomorphic correspondence} if $F_\corr(x)$ is a finite set for every $x\in X_1$.
\end{definition}   

\begin{remark}\label{Rem:explCorr}
It is helpful to encode holomorphic correspondences as holomorphic chains. Circumstances arise
where, in the notation of \eqref{E:stdForm}, $m_j\geq 2$. For instance: even if we
start with a holomorphic correspondence on $\pro$ determined by an {\em irreducible} variety
$V\subset \pro\times\pro$, composing $V$ with itself (see Section~\ref{S:keyDefns}) can result in a
variety that is not irreducible and some of whose irreducible components occur with multiplicity\,$\geq 2$.
\end{remark}  

Suppose $(X,\omega)$ is a compact K{\"a}hler manifold of dimension $k$ ($\omega$ denoting the
normalised K{\"a}hler form) and $F$ is a meromorphic correspondence  of $X$ onto itself. One of the results
in \cite{dinhSibony:dvtma06} says, roughly, that if $d_{k-1}(F)< d_k(F)$, where $d_{k-1}(F)$ and $d_k(F)$ are
the dynamical degrees of $F$ of order $(k-1)$ and $k$ respectively (see \cite[\S{3.5}]{dinhSibony:dvtma06}
for a definition of dynamical degree), then there exists a probability measure
$\mu_F$ satisfying $F^*(\mu_F) = d_k(F)\mu_F$, such that
\begin{equation}\label{E:asymp1}
 \frac{1}{d_k(F)^n}(F^n)^*(\omega^k) \weakST \mu_F \;\; \text{as measures, as $n\to \infty$.}
\end{equation}
When ${\rm dim}_\cplx(X)=1$, the assumption $d_{k-1}(F)< d_k(F)$ translates into
the assumption that the (generic) number of pre-images under $F$ is {\em strictly larger} than the
number of images under $F$, both counted according to multiplicity.
None of the techniques in the current literature are of help in studying
correspondences $F$ for which $d_{k-1}(F)\geq d_k(F)$, even when $(X,\omega)=(\mathbb{P}^k,\FS)$
(in this paper $\FS$ will denote the Fubini--Study form).
\smallskip

Why should one be interested in the dynamics of a correspondence $F:X\to X$ for which
$d_{k-1}(F)\geq d_k(F)$\,? The work of Bullett and collaborators suggest several reasons in the
case $(X,\omega)=(\pro,\FS)$. Thus, {\em we shall focus on correspondences on
$\pro$} (although parts of our results hold true for Riemann surfaces). A mating of two monic 
polynomials on $\cplx$ is a construction by Douady \cite{douady:sdh83} that, given two monic
polynomials $f, g\in \cplx[z]$ of the same degree, produces a continuous branched covering $(f\amalg g)$
of a topological sphere to itself whose dynamics emulates that of $f$ or of $g$ on separate hemispheres.
For certain natural choices of pairs $(f,g)$, one can determine in principle --- see
\cite[Theorem 2.1]{lei:mqp92} --- when $(f\amalg g)$ is semiconjugate to a rational
map on $\pro$. In a series of papers \cite{bullettPenrose:mqmmg94, bullett:ctccamp00,
bullettHarvey:mqmKgqs00, bullFrei:hcmC-lmHb05}, Bullett and collaborators extend this idea 
to matings between polynomial maps and certain discrete subgroups of the M{\"o}bius group or certain Hecke
groups. The holomorphic objects whose dynamics turn out to be conjugate to that of matings in this
new sense are holomorphic correspondences on $\pro$. Such correspondences are interesting because they
expose further the parallels between the dynamics of Kleinian groups and of rational maps. It would
be interesting to devise an ergodic theory for
such matings. {\em In all known constructions where a holomorphic correspondence
$F$ of $\pro$ models the dynamics of a mating of some polynomial with some group, $d_0(F)=d_1(F)$}.
In this context, to produce an invariant measure --- and, especially, to give an explicit prescription
for it ---  would require that the techniques in \cite{dinhSibony:dvtma06} be supplemented by other ideas.
\smallskip

We now give an informal description of our work (rigorous statements are given in
Section~\ref{S:results}). Since we mentioned Brolin's theorem, we ought to mention that an analogue
of Brolin's theorem follows from \eqref{E:asymp1} and certain other results in \cite{dinhSibony:dvtma06}
when $d_0(F)<d_1(F)$. To be more precise: there exists a polar set $\excep\varsubsetneq \pro$ such that
\begin{equation}\label{E:asymp2}
 \frac{1}{d_1(F)^n}(F^n)^*(\delta_x) \weakST \mu_F \;\; \text{as measures, as $n\to \infty$,} \;\; 
 \forall x\in \pro\setminus\excep,
\end{equation}
where $\mu_F$ is as in \eqref{E:asymp1} with $(X,\omega)=(\pro,\FS)$. This means that we have
extremely precise information about the measure $\mu_F$. Our first theorem
(Theorem~\ref{T:d_1more}) uses this information to show that the support of $\mu_F$
is disjoint from the normality set of $F$, where ``normality set'' is the analogue of the Fatou set 
in the context of correspondences.
\smallskip

When $F$ (a holomorphic correspondence on $\pro$) satisfies $d_{0}(F)\geq d_1(F)$, there is no
reason to expect \eqref{E:asymp2}. Indeed, consider these examples: $F_1(z):=1/z$, in which case
$d_0(F_1)=d_1(F_1)=1$; or the  holomorphic correspondence $F_2$ determined by the
$\pro\times\pro-$completion of the zero set of the rational function
$g(z_1,z_2)=z_2^2-(1/z_1^2)$ , in which case $d_0(F_2)=d_1(F_2)=2$.
When $d_{0}(F)\geq d_1(F)$, we draw upon certain ideas of McGehee \cite{mcgehee:acrcHs92}. We
show that if $F$ admits a repeller $\rep\subset \pro$ --- in the sense of McGehee, which extends the concept of
a repeller known for maps --- having certain
properties, then there exists a neighbourhood $U(F,\rep)\supset \rep$ and a probability measure
$\mu_F$ satisfying $F^*(\mu_F) = d_1(F)\mu_F$, such that
\begin{equation}\label{E:asymp3}
 \frac{1}{d_1(F)^n}(F^n)^*(\delta_x) \weakST \mu_F \;\; \text{as measures, as $n\to \infty$,} \;\;
 \forall x\in U(F,\rep).
\end{equation}
A rigorous statement of this is given by Theorem~\ref{T:d_1less}. The condition that $F$ admit a
repeller is very natural, and was motivated by the various examples constructed by Bullett
{\em et al}. We take up one class of these examples in Section~\ref{S:example} and show that the
conditions stated in Theorem~\ref{T:d_1less} hold true for this class. Observe that \eqref{E:asymp3} differs
from \eqref{E:asymp2} in that it does not state that $\pro\setminus U(F,\rep)$ is polar (or even nowhere
dense), but this is the best one can expect (see Remark~\ref{Rem:care} below).
\smallskip

The measure $\mu_F$ in \eqref{E:asymp2} and \eqref{E:asymp3} is not, in general,
invariant under $F$ in the measure-theoretic sense. It is merely $F^*$-invariant; see 
Section~\ref{S:results} for details. What is interesting to find is the phenomenon of equidistribution,
which arises in so many situations (see, e.g., work of Clozel, Oh and Ullmo
\cite{clozelHeeUllmo:HoeHp01}, which involves correspondences in a different context).
Yet, for a holomorphic correspondence $F$ on $\pro$ with $d_1(F)\leq d_0(F)$, we
can show that there exists a measure that {\em is} invariant under $F$ in the usual sense; see
Corollary~\ref{C:d1_less}.

\subsection*{Acknowledgements.} A part of this work was carried out during a visit by
Shrihari Sridharan to the Indian Institute of Science in 2012. He thanks the Department
of Mathematics, Indian Institute of Science, for its support and hospitality. Gautam Bharali thanks
Shaun Bullett for his very helpful answers to several questions pertaining to the example in
Section~\ref{SS:Kleinian}. The authors thank the anonymous referee of an earlier version of this
article for useful suggestions for improving the exposition herein.
\medskip   

\section{Fundamental definitions}\label{S:keyDefns}

In this section, we isolate certain essential definitions that are somewhat long. Readers who
are familiar with the rule for composing holomorphic correspondences can proceed to
Section~\ref{SS:normality}, for a definition of the normality set of a holomorphic correspondence.
\smallskip

\subsection{The composition of two holomorphic correspondences}\label{SS:compo}
Let $X$ be a complex manifold of 
dimension $k$. For any holomorphic $k$-chain $\corr$ on $X$, we define the {\em support} of 
$\corr$, assuming the representation \eqref{E:stdForm}, by
$|\corr| \ := \ \cup_{j=1}^N\Gamma_j$.
Consider two holomorphic correspondences, determined by the $k$-chains
\[
 \corr^1 \ = \ \sum_{j=1}^{N_1} m_{1,\,j}\Gamma_{1,\,j}, \qquad 
 \corr^2 \ = \ \sum_{j=1}^{N_2} m_{2,\,j}\Gamma_{2,\,j},
\]
in $X\times X$. The $k$-chains $\corr_1$, $\corr_2$ have the alternative representations
\begin{equation}\label{E:alt}
 \corr^1 \ = \ \Add{1\leq j\leq L_1}\!\IVar_{1,\,j}, \qquad
 \corr^2 \ = \ \Add{1\leq j\leq L_2}\!\IVar_{2,\,j},
\end{equation}
where the primed sums indicate that the irreducible subvarieties 
$\IVar_{s,\,j}, \ j=1,\dots,L_s$, are {\em not necessarily distinct} and are repeated according
to the coefficients $m_{s,\,j}$. Before we give the definition of $\corr^2\circ\corr^1$, observe that
we may view $\corr^1$ and $\corr^2$ as {\em relations} $|\corr^1|$ and $|\corr^2|$
on $X$. The composition-rule for relations is classical. Denoting the composition of $|\corr^1|$ and $|\corr^2|$ in
the classical sense by $|\corr^2|\star|\corr^1|$, we recall:
\begin{equation}\label{E:classic}
 |\corr^2|\star|\corr^1| \ := \ \{(x,z)\in X\times X: \exists y\in X \ \text{s.t.} \
                                (x,y)\in |\corr^1|, \ (y,z)\in |\corr^2|\}.
\end{equation}
This is the view we take in  Section~\ref{S:proof-d_1less}, where we need to make use of McGehee's
results from \cite{mcgehee:acrcHs92} on the dynamics of closed relations on compact spaces.
\smallskip

The object $\corr^2\circ\corr^1$ is, in essence, just the composition of two relations together with data that
allows one to count forward and backward images of points ``according to multiplicity''. To
begin, we define the $k$-chain $\IVar_{2,\,l}\circ\IVar_{1,\,j}$
by the conditions:
\begin{align}
 |\IVar_{2,\,l}\circ\IVar_{1,\,j}| \ := \ &\{(x,z)\in X\times X: \exists y\in X \ \text{s.t.} \
				(x,y)\in \IVar_{1,\,j}, \ (y,z)\in \IVar_{2,\,l}\}, \label{E:compos1}\\
 \IVar_{2,\,l}\circ\IVar_{1,\,j} \
 \equiv \ &\sum\nolimits_{1\leq s\leq N(j,l)}\!\nu_{s,\,jl}Y_{s,\,jl}, \notag
\end{align}
where the $Y_{s,\,jl}$'s are the distinct irreducible components of the subvariety on the right-hand 
side of \eqref{E:compos1}, and $\nu_{s,\,jl}\in \zahl_+$ is the generic number $y$'s as $(x,z)$ 
varies through $Y_{s,\,jl}$ for which the membership conditions on the right-hand side of 
\eqref{E:compos1} are satisfied. Finally, we define the $k$-chain
\begin{equation}\label{E:compos2}
 \corr^2\circ\corr^1 \ := \ \sum_{j=1}^{L_1}\sum_{l=1}^{L_2}\IVar_{2,\,l}\circ\IVar_{1,\,j}.
\end{equation}
If $\corr^1$ and $\corr^2$ determine holomorphic correspondences on $X$, then so
does $\corr^2\circ\corr^1$. It requires some amount of intersection theory to show this,
but see Section~\ref{S:basic} for an {\em elementary} proof when
$X = \pro$. The $n$-fold iterate of $\corr$ will be denoted by $\corr^{\circ n}$.
\smallskip

We adopt a notational simplification. Given a $k$-chain $\corr$ that determines
a holomorphic correspondence and there is no scope for confusion, we shall denote
$F_\corr$ by $F$. We conclude with a simple observation: for holomorphic correspondences
\begin{equation}\label{E:impo}
 |\corr^2\circ\corr^1| \ = \ |\corr^2|\star|\corr^1|.
\end{equation}

\subsection{The normality set of a holomorphic correspondence on a Riemann surface}\label{SS:normality}
Let $F$ be a holomorphic correspondence on a compact Riemann surface $\sro$. The essential
concept of the normality set of $F$ is not difficult. But we will need formalism that
enables good book-keeping. We will use the representation
\eqref{E:alt} for a holomorphic correspondence $\corr$. The set of integers
$\{m,m+1,\dots,n\}$ will be denoted by $\interR{m}{n}$.
\smallskip

Given $N\in \zahl_+$, we say that $(z_0,\dots,z_N;\,\alpha_1,\dots,\alpha_N)\in
\sro^{N+1}\times\interR{1}{L}^N$ is {\em a
path of an iteration of $F$ starting at $z_0$, of length $N$}, or simply an {\em $N$-path starting
at $z_0$}, if
\[ 
 (z_{j-1},z_j)\in \IVar_{\alpha_j}, \; \; j=1,\dots,N.
\] 
Next, given any two irreducible subvarieties $\IVar_j$ and $\IVar_k$ in the decomposition of
$\corr$ in the sense of \eqref{E:alt}, we define
\[
 \IVar_{(j,k)} \ := \
                        \{(x,y,z)\in \sro^3: (x,y)\in \IVar_{j}, \ (y,z)\in \IVar_{k}\}.
\]
This construction can be extended to any multi-index $\mul{\alpha}\in \interR{1}{L}^j$:
\begin{equation}\label{E:jcomp}
 \IVar_{\mul{\alpha}} \ := \
 \{(x_0,\dots,x_j)\in \sro^{j+1}: (x_{i-1}, x_i)\in \IVar_{\alpha_i}, \ 1\leq i\leq j\}.
\end{equation}

In all discussions on the normality set of $F$, we shall work with only those
$N$-paths $(z_0,\dots,z_N;\,\alpha_1,\dots,\alpha_N)$ that satisfy
\begin{itemize}
 \item[$(*)$] For each $j=1,\dots,N$, $\IVar_{(\alpha_1,\dots,\alpha_j)}\cap B_j$
 is an irreducible complex-analytic subvariety of $B_j$ for every sufficiently small open ball $B_j\ni (z_0,\dots,z_j)$.
\end{itemize}
An $N$-path will be called an {\em admissible $N$-path} if it 
satisfies $(*)$. Fix $z_0\in \sro$ and set
\[
 \path_N(z_0) \ := \ \text{the set of all paths of iterations of $F$, of length $N$,
                        starting at $z_0$,}
\]
$N\in \zahl_+$. We will denote an element of $\path_N(z_0)$ by $\Pth$.
Observe that if $\Pth$ is an admissible $N$-path, $N\geq 2$, then there is a {\em unique} irreducible
component of $\IVar_{(\alpha_1,\dots,\alpha_j)}$ to which $(z_0,\dots,z_j)$ belongs, $j=2,\dots,N$.
Hence, if $\Pth$ is an admissible $N$-path, let us write
\[
 \IVar_{\Pth,\,j} \ := \ \begin{cases}
			\IVar_{\alpha_1}, &\text{if $j=1$}, \\
			\text{\small{the unique irreducible component of 
			$\IVar_{(\alpha_1,\dots,\alpha_j)}$
			containing $(z_0,\dots,z_j)$}}, &\text{if $j\geq 2$}.
			\end{cases}
\]
Let $\nu_{(\Pth,\,j)}:Y_{\Pth,\,j}\lrarw \IVar_{\Pth,\,j}$, where $Y_{\Pth,\,j}$
is a compact Riemann surface, denote the desingularization of $\IVar_{\Pth,\,j}$.  We now have the
essential notations needed to define the normality set. The definitions that follow are strongly
influenced by the notion introduced by Bullett and Penrose \cite{bullettPenrose:rlshc01}. Yet, what
we call a ``branch of an iteration'' will look vastly different from its namesake in
\cite{bullettPenrose:rlshc01}. This is because, for our purposes, we will need to label (resp., track)
all the maps involved by (resp., along) the various paths that comprise $\path_N(z_0)$.
I.e., the difference is (largely) in formalism.
The one alteration that we make to the Bullett--Penrose 
definition is in working with only admissible paths: this allows us to parametrise holomorphic branches of the
iteration of $F$ along $\Pth$ --- which we shall presently define --- simply by $\Pth$ itself 
(see Remark~\ref{Rem:horrors} as well).
\smallskip

Let $\prjc_k$ and $\pi^j_k$ denote the following projections:
\begin{align}
 \prjc_k:\sro^{k+1}\lrarw \sro^{k}, \; \; \prjc_k&:
                (x_0,\dots,x_k)\longmapsto (x_0,\dots,x_{k-1}), \notag \\
 \pi^j_k:\sro^{k+1}\lrarw \sro, \; \; \pi^j_k&:
                (x_0,\dots,x_k)\longmapsto x_j, \; \; 0\leq j\leq k, \notag
\end{align}
where $k\in \zahl_+$. The idea of a holomorphic branch of an iteration of $F$ along $\Pth$,
$\Pth\in \path_N(z_0)$ and admissible, is to assign to a suitable
open neighbourhood $U\ni z_0$ a finite sequence $(U(\Pth,j))_{1\leq j\leq N}$
of $\IVar_{\Pth,\,j}$-open sets. In the case $j = 1$, the natural choice of $U(\Pth,1)$ is
the irreducible component of  $\IVar_{\Pth,\,1}\cap(U\times X)$ containing $(z_0, z_1)$. Proceed
inductively and set: 
\begin{align*}
 U(\Pth,j) \ &= \ \text{\small{the irreducible component of $\OM(\Pth,j)\cap\IVar_{\Pth,\,j}$ containing
 $(z_0,\dots,z_j)$, and}} \\
 \OM(\Pth,j) \ &:= \ \begin{cases}
 				 U\times\sro, &\text{if $j=1$}, \\
 				 \left[
 				 U\times \prod_{k=1}^{j-1}\pi_k^k(U(\Pth,k))\right]\times\sro, &\text{if $j\geq 2$.}
 				 \end{cases}
\end{align*}
The phrase ``irreducible component'' above signifies that
$\OM(\Pth,j)\cap\IVar_{\Pth,\,j}$ be viewed as a complex-analytic subvariety of $\OM(\Pth,j)$.
Now, if each $U(\Pth,j)$ were obtainable as the graph of some holomorphic map
$\Psi_j : U\lrarw \cplx^{j}$, then the sequence $(\Psi_j)_{1\leq j\leq N}$ would be a natural
definition of a {\em holomorphic} branch along $\Pth$; see Remark~\ref{Rem:illust}. 
However, if some  $U(\Pth,j)$ has a singularity at $(z_0,\dots,z_j)$, or does not project
injectively under $\pi^0_j$, then the preceding idea does not work. Therefore, we modify this
idea by substituting the aforementioned $(\Psi_j)_{1\leq j\leq N}$ by a sequence of natural
parametrisations of the $U(\Pth,j)$'s.
\smallskip

To this end, fix a $z_0\in \sro$ and
an {\em admissible} path $\Pth\in \path_N(z_0)$. Each $U(\Pth,j)$ is parametrised via
$\nu_{(\Pth,\,j)}$ by an open patch in the Riemann surface $Y_{\Pth,\,j}$. We say that
{\em there exists a holomorphic branch of an iteration of $F$ on $U$ along $\Pth$}, which we
denote by $(\psi_{(\Pth,\,1)},\dots,\psi_{(\Pth,\,N)};\,U)$, if there exist a connected neighbourhood
$U$ of $z_0$ and holomorphic mappings $\psi_{(\Pth,\,j)}:\dom\lrarw Y_{(\Pth,\,j)}$ of
a planar domain $\dom$ containing $0$ that, for each $j=1,\dots,N$, satisfy three conditions:
\begin{itemize}
 \item[1)] $\nu_{(\Pth,\,j)}\circ \psi_{(\Pth,\,j)}(0)=(z_0,\dots,z_j)$.
 \item[2)] (Compatibility condition) $\prjc_{j}\circ\nu_{(\Pth,\,j)}\circ\psi_{(\Pth,\,j)}=
 \nu_{(\Pth,\,j-1)}\circ\psi_{(\Pth,\,j-1)}$, $j\neq 1$.
\end{itemize}
A part of our third condition will encode the requirement that $U(\Pth,1),\dots, U(\Pth,N)$ contain no
singularities at which they are locally reducible (i.e., as analytic germs). (Loosely speaking, this ensures that
any point $z_*\neq z_0$ sufficiently close to $z_0$ will have a holomorphic branch of an iteration of $F$ along
{\em some} admissible path in $\path_N(z_*)$ that is ``sufficiently close'' to
$(\psi_{(\Pth,\,1)},\dots,\psi_{(\Pth,\,N)})$ --- see Remark~\ref{Rem:open}.) More precisely:
\begin{itemize}
 \item[3)] $\psi_{(\Pth,\,j)}$ is a finite-sheeted (perhaps branched) covering map onto 
 $\nu_{(\Pth,\,j)}^{-1}(U(\Pth,j))$, and $\nu_{(\Pth,\,j)}$ maps the latter
 set {\bf homeomorphically onto} $U(\Pth,j)$.
\end{itemize}

\begin{remark}\label{Rem:illust}
Note that if $F$ is a non-constant rational {\em map} on $\pro$, then the maps
\[
 \psi_{(\Pth,\,j)} \ := \
 \left({\sf id}_{\dom}+z_0,\left.F(\bcdot+z_0)\right|_{\dom},\dots,
 \left.F^{j-1}\circ F(\bcdot+z_0)\right|_{\dom}\right), \; \; \; j=1,2,3,\dots,
\]
where $\dom$ is a small disc around $0$, satisfy all the conditions above (taking
each $Y_{(\Pth,\,j)}$ to be the graph of the appropriate $(\pro)^j$-valued map). 
\end{remark}

\noindent{Having defined holomorphic
branches, we can give the following definition.}

\begin{definition}\label{D:NSet}
Let $\sro$ be a compact Riemann surface and let
$F$ be a holomorphic correspondence on $\sro$. A point $z_0\in \sro$ is said to belong
to the {\em normality set of $F$}, denoted by $\norm(F)$, if there exists a connected
neighbourhood $U$ of $z_0$ and a {\bf single} planar domain $\dom$ containing $0$,
which depends on $z_0$, such that
\begin{itemize}
 \item[1)] For each $n\in \zahl_+$ and each $\Pth\in \path_n(z_0)$, there exists
 a holomorphic branch $(\psi_{(\Pth,\,1)},\dots,\psi_{(\Pth,\,n)};\,U)$ of an iteration of $F$ along $\Pth$
 with ${\sf Dom}(\psi_{(\Pth,\,j)})=\dom$ for every $(\Pth,\,j)$.
 \item[2)] The family
 \begin{align*}
 \Fml(z_0) \ := \ &\left\{\pi_n^n\circ\nu_{(\Pth,\,n)}\circ\psi_{(\Pth,\,n)}: n\in \zahl_+,\;
                        \Pth\in \path_n(z_0), \ \text{\& $(\psi_{(\Pth,\,1)},\dots,\psi_{(\Pth,\,n)};\,U)$}\right. \\
                &\qquad\left.\text{is a holomorphic branch of an iteration of $F$ along $\Pth$}\right\}
 \end{align*}
 is a normal family on $\dom$.
\end{itemize}
\end{definition}

\begin{remark}\label{Rem:open}
The set $\norm(F)$ is open, although it is not necessarily non-empty. If $z_0\in \norm(F)$ and $U$ is
the neighbourhood of $z_0$ as given by Definition~\ref{D:NSet}, then it is routine to show that
$U\subset \norm(F)$.
\end{remark}

\section{More definitions and statement of results}\label{S:results}

We need to present some constructions before we can state our first result.
\smallskip
   
Given a holomorphic correspondence on $X$, ${\rm dim}_{\cplx}(X)=k$, determined by a
holomorphic $k$-chain $\corr$, its
{\em adjoint correspondence} is the meromorphic correspondence determined by the $k$-chain
(assuming the representation \eqref{E:stdForm} for $\corr$)
\[
 \acorr \ := \ \sum_{j=1}^{N} m_j\aGa{j},
\]
where $\aGa{j} := \{(y,x)\in X\times X: (x,y)\in \Gamma_j\}$.
In general, $\acorr$ may not determine a holomorphic correspondence. However,
when ${\rm dim}_{\cplx}(X) = 1$, it is easy to see that {\em any meromorphic correspondence
of $X$ is automatically holomorphic}. Thus, if $F_{\corr}$ is a holomorphic correspondence on
$\pro$, then so is $F_{\acorr}$. In the abbreviated
notation introduced in Section~\ref{S:keyDefns}, we shall henceforth write:
\[
 F^n \ := \ F_{\corr^{\circ n}}, \qquad \eff \ := \ F_{\acorr}.
\]

Given a holomorphic $k$-chain $\corr$ on $X\times X$, $\corr$ detemines a current of 
bidimension $(k,k)$ via the currents of integration given by its constituent subvarieties
$\Gamma_j$. We denote this current by $[\corr]$. If $F$ is
determined by $\corr$, we {\em formally} define the action of $F$ on currents $S$ on $X$
of bidegree $(p,p)$, $0\leq p\leq k$, by the prescription:
\begin{equation}\label{E:pullback}
 F^*(S) \ := \ (\pi_1)_*\left(\pi_2^*(S)\wedge [\corr]\right),
\end{equation}
where, as usual, $\pi_j$ denotes the projection of $X\times X$ onto the $j$th factor.
This prescription will make sense for those currents for which the pullback by $\pi_2$ makes
sense and the intersection of this new current with $[\corr]$ also makes sense. That this is the
case is easy to see when $S$ is a smooth $(p,p)$ form (hence a current of bidegree $(p,p)$ on $X$).
The reader is referred to \cite[Section 2.4]{dinhSibony:dvtma06} for details.
\smallskip

A finite Borel measure $\mu$ on $X$ can be viewed as a current of bidegree $(k,k)$.
Let us work out $F^*(\mu)$ for a specific example that is central to this paper. Let $x\in X$
and let $\delta_x$ be the Dirac mass at $x$. The prescription \eqref{E:pullback} is interpreted as
\begin{align}
 \langle F^*(\delta_x), \varphi \rangle\,\equiv_{{\rm by \ duality}}\,\left\langle
 	\pi_2^*(\delta_x)\wedge [\corr], \pi_1^*\varphi \right\rangle \
 	:=& \ \sum_{j=1}^N m_j\langle(\left.\pi_2\right|_{\Gamma_j})^*(\delta_x),
								\pi_1^*\varphi \rangle \notag \\
	=& \ \sum_{j=1}^N m_j\langle \delta_x, \big(\left.\pi_2\right|_{\Gamma_j}\big)_*
								(\pi_1^*\varphi) \rangle, 
		\label{E:pullbackDirac1}
\end{align}
where each summand in the last expression is just the way one defines the pullback
of a current under a holomorphic mapping (in this case, $\left.\pi_2\right|_{\Gamma_j}, \
j=1,\dots, N$) of an analytic space that is submersive on a Zariski open subset. If $\OM\subset X$ is
a Zariski-open subset of $X$ such that $(\pi_2^{-1}(\OM)\cap\Gamma_j,\OM,\pi_2)$ is a covering
space for each $j=1,\dots,N$, then, for $x\in \OM$, $(\left.\pi_2\right|_{\Gamma_j})_*(f)(x)$ is
just the sum of the values of $f$ on the fibre $\pi_2^{-1}\{x\}\cap\Gamma_j$ for any 
$f\in \smoo^\infty(X\times X)$. Thus, when $x\in \OM$, \eqref{E:pullbackDirac1} equals the quantity
\begin{equation}\label{E:pullbackDirac2}
 \sum_{j=1}^N m_j\negthickspace\negthickspace\sum_{\zt:(\zt,x)\in\Gamma_j}\negthickspace\varphi(\zt) \
 =: \ \Lam[\varphi](x) \qquad x\in \OM.
\end{equation}
For any fixed {\em continuous} function $\varphi$, $\Lam[\varphi]$ extends
continuously to each $x\in X\setminus\OM$. We shall denote this continuous extension
of the left-hand side of \eqref{E:pullbackDirac2} also as  $\Lam[\varphi]$. In other words,
$F^*(\delta_x)$ can be defined as a measure supported on the set 
$\eff(x)$, and
\begin{equation}\label{E:pullbackDirac3}
 \langle F^*(\delta_x), \varphi \rangle \ = \ \Lam[\varphi](x) \; \; \forall x\in X, \; \;
						\forall \varphi\in \smoo(X).
\end{equation}
The arguments preceding \eqref{E:pullbackDirac2} continue to be valid if, in \eqref{E:pullbackDirac1},
$\delta_x$ is replaced by $\mu$, a finite, positive Borel measure on $X$.
\smallskip 

The push-forward of a current $S$ by $F$ is defined by the equation $F_*(S):=(\eff)^*(S)$ whenever the
latter makes sense.
\smallskip

We define two numbers that are essential to the statement of our theorems. With $F$ as
above, let $d_1(F)$ denote the generic number of preimages under
$F$ of a point in $\pro$, counted according to multiplicity. More rigorously, this
means --- $\OM$ being any Zariski-open set of the type described prior to the equation
\eqref{E:pullbackDirac2} --- that
\begin{equation}\label{E:topo_d}
 d_1(F) \ = \ \sum_{j=1}^Nm_j{\sf Card}\{x:(x,y)\in \Gamma_j\}, \quad y\in \OM,
\end{equation}
which is independent of the choice of $y\in \OM$. In other words, $d_1(F)$ is the
{\em topological degree of $F$}, often denoted as $d_{top}(F)$. Define
$d_0(F):=d_{top}(\eff)$.
\smallskip

We will first consider a holomorphic correspondence $F$ of $\pro$ such that $d_1(F) > d_0(F)$.
A very special case of a result of Dinh and Sibony is that
there exists a probability measure $\mu_F$ such that
\begin{equation}\label{E:asympFS}
 \frac{1}{d_1(F)^n}(F^n)^*(\FS) \weakST \mu_F \;\; \text{as measures, as $n\to \infty$,}
\end{equation}
where $\FS$ denotes the Fubini--Study form on $\pro$, treated as a normalised area form.
Let us call this measure the {\em Dinh--Sibony measure associated to $F$}. Since
equidistribution is among the themes of this paper, we should mention
that for a generic $z\in \pro$, $\mu_F$ is the asymptotic distribution of the iterated
pre-images of $z$. More precisely:

\begin{fact}\label{L:clarify}
Let $F$ be a holomorphic correspondence on $\pro$ such that $d_0(F)< d_1(F)$
and let $\mu_F$ be the Dinh--Sibony measure associated to $F$. There exists a 
polar set $\excep\varsubsetneq \pro$ such that for each $z\in \pro\setminus\excep$ 
 \[
 \frac{1}{d_1(F)^n}(F^n)^*(\delta_z) \weakST \mu_F \;\; \text{as measures, as $n\to \infty$}.
 \]
Consequently, $F^*(\mu_F)=d_1(F)\mu_F$.
\end{fact}

\noindent{The above follows by combining \eqref{E:asympFS} with another result from
\cite{dinhSibony:dvtma06}.
Fact~\ref{L:clarify} will have {\em no role to play} herein
except to set the context for our results. For instance, it shows that ${\sf supp}(\mu_F)$ could be viewed --- especially
in view of Brolin's results --- as the analogue of the Julia set in the general context of correspondences. While,
for a rational map on $\pro$,
its Julia set is {\em definitionally} disjoint from its normality set, in the case of correspondences we have:}

\begin{theorem}\label{T:d_1more}
Let $F$ be a holomorphic correspondence on $\pro$ and assume that $d_0(F)< d_1(F)$.
Let $\mu_F$ be the Dinh--Sibony measure associated to $F$. Then,  ${\rm supp}(\mu_F)$ is
disjoint from the normality set of $F$.
\end{theorem}

The ideas behind the proof are as follows. For $\norm(F)\neq \varnothing$,
we shall show that one can apply Marty's normality criterion in such a manner as to deduce that the
volumes of any compact $K\Subset \norm(F)$ given by the measures induced by $(F^n)^*(\FS)$
are bounded independent of $n$. The result then follows due to the fact that
$d_1(F)^{-n}\to 0$ as $n\to \infty$.  
\smallskip

The situation is {\em very different} when $d_0(F)\geq d_1(F)$. To repeat: we
should not expect asymptotic equidistribution of preimages in general, even 
when $d_0(F), d_1(F)\geq 2$, 
as the holomorphic correspondence $F$ whose graph is the $\pro\times\pro-$completion of the 
zero set of the rational function $g(z_1,z_2)=z_2^2-(1/z_1^2)$ illustrates. We require
some dynamically meaningful condition for things to work. It is this
need that motivates the next few definitions. Let $X$ be a compact Hausdorff space and
let $f\subset X\times X$ be a relation of $X$ to itself such that $\pi_1(f)=X$.
For any set $S\subset X$, we write
\[
 f(S) \ := \ \pi_2\left(\pi_1^{-1}(S)\cap f\right).
\]
We define the $n$th iterated relation by
\begin{align}
 f^{(n)} \ &:= \ f\star f^{(n-1)} \; \; \text{for $n\geq 2$}, \label{E:classic2} \\
 f^{(1)} \ &:= \ f, \notag
\end{align}
where the composition operation $\star$ is as given by \eqref{E:classic} above. 
It is useful to have a notion of omega limit sets in the context of iterating a relation
analogous to the case of maps. This definition is provided
by McGehee in \cite[Section 5]{mcgehee:acrcHs92}. Following McGehee, for a subset 
$S\subset X$, let us write
\[
 \mathfrak{K}(S;f) \ := \{K\subset_{\rm closed}X: K \ \text{satisfies $f(K)\subset K$ and 
				$f^{(n)}(S)\subset K$ for some $n\geq 0$}\}
\]
(with the understanding that $f^{(0)}$ is the diagonal in $X\times X$). The
{\em omega limit set of $S$ under $f$}, denoted by $\omega(S;f)$, is the set
\[
 \omega(S;f) \ := \ \bigcap\mathfrak{K}(S;f).
\]
We say that a set $\mathcal{A}\subset X$ is an {\em attractor} for $f$ if 
$\mathcal{A}\neq X$ and there
exists a set $U$ such that $\mathcal{A}\subset U^\circ$ and such that 
$\omega(U;f)=\mathcal{A}$.
\smallskip

These concepts motivate the following two definitions in the context of holomorphic correspondences.

\begin{definition}\label{D:repeller}
Let $F$ be a holomorphic correspondence on a Riemann surface $X$ given by the holomorphic 
$1$-chain $\corr$. A set $\mathcal{A}\subset X$ is called an {\em attractor} for $F$ if it is an attractor
for the relation $|\corr|$ in the sense of \cite{mcgehee:acrcHs92} (i.e., as discussed above). A set 
$\rep$ is called a {\em repeller} for $F$ if it is an attractor for the relation $|\acorr|$.
\end{definition}

\noindent{We must note here McGehee calls the relation on $X$ induced by $|\acorr|$ the {\em transpose} 
of $|\corr|$, and our $|\acorr|$ is $|\corr|^*$ in the notation of \cite{mcgehee:acrcHs92}.}

\begin{definition}\label{D:strongRep}
Let $F$ be as above and let $\rep$ be a repeller for $F$.
We say that $\rep$ is a {\em strong repeller} for $F$ if there exists a point $a_0\in \rep$
and an open set $U\supset \rep$ such that for each $w\in U$, there exists a sequence
$\{a_n(w)\}_{n\in \zahl_+}$ such that
\begin{itemize}
 \item $a_n(w)\in (\eff)^n(w) \ \forall n\in \zahl_+$; and
 \item $a_n(w)\lrarw a_0$ as $n\to \infty$.
\end{itemize}
The term {\em strong attractor} has an analogous definition.
\end{definition} 

We
call $w\in \pro$ a {\em critical value} if there exists an irreducible component $\Gamma_j$ such that 
at least one of the irreducible germs of $\Gamma_j$ at some point in $\pi_2^{-1}\{w\}\cap\Gamma_j$
is either non-smooth or does {\em not} project injectively under $\pi_2$.
\smallskip

We are now in a position to state our next result.

\begin{theorem}\label{T:d_1less}
Let $F$ be a holomorphic correspondence on $\pro$ for which $d_0(F)\geq d_1(F)$. Assume that
$F$ has a strong repeller $\rep$ that is disjoint from the set of critical values of $F$. Then,
there exist a probability measure $\mu_F$ on $\pro$ that satisfies
$F^*(\mu_F)=d_1(F)\mu_F$ and an open set $U(F,\rep)\supset \rep$ 
such that
\begin{equation}\label{E:equi}
 \frac{1}{d_1(F)^n}(F^n)^*(\delta_z) \weakST \mu_F \;\; \text{as measures, as $n\to \infty$} \;\;
 \forall z\in U(F,\rep).
\end{equation}
\end{theorem}
\smallskip

\noindent{It may seem to the reader that \eqref{E:equi} could be stronger, since the theorem
does not state that $\pro\setminus U(F,\rep)$ is polar (or even nowhere dense). However,
given that $d_0(F)\geq d_1(F)$, this is very much in the nature of things. In this regard, we
make the following remark.}

\begin{remark}\label{Rem:care}
If $F$ is as in Theorem~\ref{T:d_1less}, we {\em cannot conclude, in general,
that the set $\pro\setminus U(F,\rep)$ is polar.} The following example constitutes
a basic obstacle to $\pro\setminus U(F,\rep)$ being even nowhere dense. 
Let $P$ be any polynomial whose filled Julia set has non-empty
interior. Consider the holomorphic correspondence $F$ determined by
\[
 \corr^P \ := \ \text{the completion in $\pro\times\pro$ of the zero set of $(z_1-P(z_2))$.}
\]
Here, $d_0(F)=\deg(P)> 1= d_1(F)$. Note that $\{\infty\}$ is a strong repeller. However,
$U(F,\{\infty\})$ cannot contain any points from the filled Julia set of $P$.
\end{remark}

From the perspective of studying the problem of equidistribution of {\em inverse}
images, Theorem~\ref{T:d_1less} is, to the best of our knowledge, the first theorem
concerning the equidistribution problem for holomorphic correspondences $F$ such
that $d_0(F)\geq d_1(F)$.
\smallskip

We must point out that, in general, given $d_1(F) > d_0(F)$, the Dinh--Sibony measure
is {\em not} an invariant measure. (It is what one calls an $F^*$-invariant measure.)
A Borel measure $\mu$ is said to be {\em invariant under $F$}
if its push-forward by $F$ preserves (compensating for multiplicity if $F$ is not a map) measure of
all Borel sets --- i.e., if $F_*\mu = d_0(F)\mu$. Thus, under the condition $d_0(F)\geq d_1(F)$,
we can actually construct measures that are invariant under $F$:

\begin{corollary}\label{C:d1_less}
Let $F$ be a holomorphic correspondence on $\pro$ for which $d_0(F)\geq d_1(F)$.
\begin{enumerate}
 \item[$i)$] If $d_0(F) > d_1(F)$, there exists a measure $\mu_F$ that is invariant
 under $F$.
 \item[$ii)$] Suppose $d_0(F) = d_1(F)$. If $F$ has a strong attractor that is disjoint
 from the critical values of $\eff$, then there exists a measure $\mu_F$ that is invariant
 under $F$.
\end{enumerate}
\end{corollary} 

The proof of Theorem~\ref{T:d_1less} relies on techniques developed by Lyubich in 
\cite{lyubich:epreRs83}. Given our hypothesis on the existence of a repeller $\rep$, one
can show that there exists a compact set $B$ such that $\rep\subset B^\circ$ and
$\eff(B)\subset B$. This allows us to define a Perron--Frobenius-type operator
$\PerrF_B:\smoo(B;\cplx)\lrarw \smoo(B;\cplx)$, where
\[
 \PerrF_B \ := \ \frac{1}{d_1(F)}\left.\Lam\right|_B,
\]
with $\left.\Lam\right|_B$ being the operator given by \eqref{E:pullbackDirac2} with $B$
replacing $\Omega$. Our proof relies on showing
that the family $\{\PerrF_B^n:n=1,2,3,\dots\}$ satisfies the conditions of the main result
in \cite[\S{2}]{lyubich:epreRs83} (which provides us a candidate for $\mu_F$).
This goal is achieved, {\em in part}, by showing that
there exists an open neighbourhood $W$ of the $B$ above such that $\eff(W)\subset W$,
and that around each $z\in W$ and for each $n\in \zahl_+$, there are $d_1(F)^n$
holomorphic branches (counting according to multiplicity) of $n$-fold iteration of the
correspondence $\eff$.
\smallskip

The examples in Remark~\ref{Rem:care} have some very
particular features. One might ask whether there are plenty of holomorphic correspondences $F$ on
$\pro$ with $d_0(F)\geq d_1(F)$ --- {\em without} such special features as in the correspondences
in Remark~\ref{Rem:care} --- that satisfy the conditions stated in
Theorem~\ref{T:d_1less}. One might also ask whether any of the correspondences alluded to in
Section~\ref{S:intro} satisfy the conditions in Theorem~\ref{T:d_1less}. 
The reader is referred to Section~\ref{S:example} concerning these questions.
In the next section, we shall establish a few technical facts which will be of relevance throughout
this paper. The proofs of our theorems will be provided in Sections~\ref{S:proof-d_1more} and
\ref{S:proof-d_1less}.
\medskip

\section{Technical propositions}\label{S:basic} 

We begin by showing that the composition of two holomorphic correspondences on $\pro$, under the 
composition rule \eqref{E:compos2}, produces a holomorphic correspondence. 
\smallskip

One way to see this is to begin with how one computes $F_2\circ F_1$ if
one is given exact expressions for $F_1$ and $F_2$. 
Let $\corr^s$ be the graph of $F_s$, $s=1,2$, and consider the representations given
by \eqref{E:alt}. Fix indices $j$ and $l$ such that $1\leq j\leq L_1$ and
$1\leq l\leq L_2$. It follows that there exist irreducible polynomials $P_1, P_2\in \cplx[z,w]$ such
that
\[
 \IVar_{1,\,j}\cap\CC = \{(z,w)\in \CC: P_1(z,w)=0\}, \quad \IVar_{2,\,l}\cap\CC = \{(z,w)\in \CC: 
P_2(z,w)=0\};
\]
see, for instance, \cite[pp.\,23-24]{shafarevich:bagI94}. Now, given any polynomial $P\in \cplx[z,w]$,
set
\begin{align}
 {\rm supp}(P) \ &:= \ \{(\alpha,\beta)\in \nat^2:\bdy^\alpha_z\bdy^\beta_w P(0,0)\neq 0\}, \notag \\
 d_z(P) \ &:= \ \max\{\alpha\in \nat: (\alpha,\beta)\in {\rm supp}(P) \}, \notag \\
 d_w(P) \ &:= \ \max\{\beta\in \nat: (\alpha,\beta)\in {\rm supp}(P) \}. \notag   
\end{align}
Then, there is a choice of projective coordinates on $\pro$ such that
\begin{align}
 \IVar_{1,\,j} \ = \ \{([z_0:z_1],[w_0:w_1])\in \pro\times\pro: 
		z_0^{d_z(P_1)}w_0^{d_w(P_1)}P_1(z_1/z_0,w_1/w_0)=0\}, \notag \\
 \IVar_{2,\,l} \ = \ \{([z_0:z_1],[w_0:w_1])\in \pro\times\pro: 
                z_0^{d_z(P_2)}w_0^{d_w(P_2)}P_2(z_1/z_0,w_1/w_0)=0\}. \notag
\end{align}
With these notations, we are in a position to state our first proposition.

\begin{proposition}\label{P:noLines}
Let $\IVar_{1,\,j}$ and $\IVar_{2,\,l}$ be irreducible subvarieties belonging to the 
holomorphic $1$-chains $\corr^1$ and $\corr^2$ respectively. Let $P_1$ and $P_2$ be the
defining functions of $\IVar_{1,\,j}\cap\CC$ and $\IVar_{2,\,l}\cap\CC$ respectively.
\begin{enumerate}
 \item[$i)$] Let $R(z,w):={\sf Res}(P_1(z,\bcdot),P_2(\bcdot,w))$, where ${\sf Res}$ denotes
 the resultant of two univariate polynomials. Let $V_R$ denote the biprojective completion
 in $\pro\times\pro$ of $\{(z,w)\in\CC:R(z,w)=0\}$. Then $|\IVar_{2,\,l}\circ\IVar_{1,\,j}|=V_R$.
 \item[$ii)$] $V_R$ has no irreducible components of the form $\{a\}\times\pro$ or
 $\pro\times\{a\}$, $a\in \pro$.
\end{enumerate}
\end{proposition}
\begin{proof} Let us write $V:=|\IVar_{2,\,l}\circ\IVar_{1,\,j}|$. Since two polynomials 
$p,q\in \cplx[X]$ have a common zero if and only if ${\sf Res}(p,q)=0$,
\[
 V\cap\CC \ = \ \{(z,w)\in \CC: {\sf Res}(P_1(z,\bcdot),P_2(\bcdot,w))=0\}.
\]
Hence, as $V$ is the biprojective completion of $V\cap\CC$ in $\pro\times\pro$, $(i)$ follows.
\smallskip

To prove $(ii)$, let us first consider the case when $a\neq [0:1]$. Then, it suffices to
show that $R$ has no factors of the form $(z-a)$ or $(w-a)$. We shall show that
$R$ has no factors of the form $(z-a)$. An analogous argument will rule
out factors of the form $(w-a)$. To this end, assume that there exists
an $a\in \cplx$ such that $(z-a)|R$ in $\cplx[z,w]$. This implies
\[
 R(a,w) \ = \ 0 \;\; \forall w\in \cplx.
\]
Thus, for each $w\in \cplx$, the polynomial $P_2(\bcdot,w)$ has a zero in common with
$p_1:=P_1(a,\bcdot)\in \cplx[X]$. Note that $p_1\not\equiv 0$ because, otherwise,
$(z-a)|P_1$, which would contradict the fact that $\left.\pi_s\right|_{\IVar_{1,\,j}}$ is surjective,
$s=1,2$. Thus, there exists an uncountable set $S\subset \cplx$ and a point $b \in p_1^{-1}\{0\}$
such that
\[
 P_2(b,w) \ = \ 0 \;\; \forall w\in S.
\]
But this implies $P_2(b,\bcdot)\equiv 0$, i.e. that $(z-b)|P_2$. This is impossible, for
exactly the same reason that $(z-a)\!\nmid\!P_1$. Hence $R$ has no factors of the form 
$(z-a), \ a\in\cplx$.
\smallskip

Note that, if we write $\cplx^\prime:=\{[z_0:z_1]\in \pro:z_1\neq 0\}$, then, arguing as in the beginning
of this proof, 
\[
 V\cap(\cplx^\prime\times\cplx) \ = \ 
\{(z,w)\in \CC: z^{d_z(R)}{\sf Res}(P_1(1/z,\bcdot),P_2(\bcdot,w))=0\},
\]
where $d_z(R)$ is as defined in the beginning of this section. If we define $\widetilde{R}\in \cplx[z,w]$ by
\[
 \widetilde{R}(z,w) \ := \ z^{d_z(R)}{\sf Res}(P_1(1/z,\bcdot),P_2(\bcdot,w)),
\]
then we get $z\!\nmid\!\widetilde{R}$ in $\cplx[z,w]$. Thus,
$\{[0:1]\}\times\pro$ is not an irreducible component of $V$. By a similar argument,
$\pro\times\{[0:1]\}$ is not an irreducible component of $V$ either.
\end{proof}

It is now easy to see that $\corr^2\circ\corr^1$ determines a holomorphic correspondence on $\pro$.
Let us pick $\IVar_{1,\,j}$ and $\IVar_{2,\,l}$ as in Proposition~\ref{P:noLines}
and let $C$ be an irreducible component of $|\IVar_{2,\,l}\circ\IVar_{1,\,j}|$. By the fundamental
theorem of algebra, $\left.\pi_s\right|_C$ would fail to be surjective for some $s\in \{1, 2\}$ only
if $C$ is of the form $\{a\}\times\pro$ or $\pro\times\{a\}$, $a\in \pro$. This is impossible by
Part~$(ii)$ of Proposition~\ref{P:noLines}. Hence, we have the following: 
  
\begin{corollary}\label{C:allOkay}
Let $\corr^1$ and $\corr^2$ be two holomorphic correspondences on $\pro$. Then
$\corr^2\circ\corr^1$ is a holomorphic correspondence on $\pro$.
\end{corollary}

The next lemma will be useful in simplifying expressions of the form 
$(F^n)^*\delta_z$ or $(F^n)_*\delta_z$. Its proof is entirely routine, so we shall leave the
proof as an exercise.

\begin{lemma}\label{L:iterate}
Let $X$ be a compact complex manifold and let $F$ be a holomorphic correspondence
on $X$. Then $\trp{(F^n)}=(\eff)^n \ \forall n\in \zahl_+$.
\end{lemma}

The final result in this section is important because it establishes that the measures
$d_1(F)^{-n}(F^n)^*(\delta_z)$ appearing in Theorem~\ref{T:d_1less}
are probability measures. 

\begin{proposition}\label{P:degreeSeq}
Let $F$ be a holomorphic correspondence on $\pro$. Then $d_{top}(F^n)=d_{top}(F)^n \ 
\forall n\in \zahl_+$.
\end{proposition} 

\noindent{The result above is obvious if $F$ is a map. Essentially the same reasoning applies
when $F$ is not a map. In the definition of the topological degree, any
point $y\in \pro$ that satisfies equation \eqref{E:topo_d} is generic. Since the correspondence
$\corr$ (which determines $F$) has no irreducible components of the form
$\{a\}\times\pro$ or $\pro\times\{a\}$, $a\in \pro$, all the preimages of each such $y\in \pro$ --- with,
perhaps, the exception of finitely many points --- is generic in the sense of \eqref{E:topo_d}.
Hence, $d_{top}$ is multiplicative under composition of correspondences.}
\medskip

\section{The proof of Theorem~\ref{T:d_1more}}\label{S:proof-d_1more}

We begin this section with some remarks on our use of the results of Dinh--Sibony
\cite{dinhSibony:dvtma06}. The result we shall use is the one that leads to \eqref{E:asympFS}.
The precise result is \cite[Th{\'e}or{\`e}me~5.1]{dinhSibony:dvtma06}, read together with
Remarques~5.2.
We state a version below {\em specifically for correspondences on $\pro$.} We leave it to the reader to verify
our transcription of \cite[Th{\'e}or{\`e}me~5.1]{dinhSibony:dvtma06} to the present context. It might
be helpful for readers who are unfamiliar with \cite{dinhSibony:dvtma06} to mention a
couple of identities needed for this transcription. We shall not define here the notion of
{\em intermediate degrees of $F$ of order $s$}; we just refer to \cite[Section~3.1]{dinhSibony:dvtma06}. 

\begin{fact}\label{L:topDeg}
Let $(X,\omega)$ be a compact K{\"a}hler manifold of dimension $k$ and assume
$\int_X\omega^k = 1$. Let $F:X\lrarw X$ be a holomorphic correspondence. Let
$\lambda_s(F)$ denote the intermediate degree of $F$ of order $s$, $s=0, 1,\dots,k$.
Then:
\begin{itemize}
 \item[$i)$] $\lambda_{k-s}(\eff) = \lambda_s(F)$.
 \item[$ii)$] $\lambda_k(F) = d_{top}(F)$.
\end{itemize}
\end{fact}

In what follows $\FS$ shall denote the Fubini--Study form normalised so that $\int_{\pro}\FS = 1$.
The key result needed is:

\begin{result}[Th{\'e}or{\`e}me~5.1 of \cite{dinhSibony:dvtma06} paraphrased for
correspondences of $\pro$]\label{R:1stRes}
Let $F_n$, $n=1,2,3,\dots$, be holomorphic correspondences of $\pro$. Suppose that the series\linebreak 
$\sum_{n\in \zahl_+}(d_0(F_1)/d_1(F_1))\dots(d_0(F_n)/d_1(F_n))$ 
converges. Then, there exists a probability measure $\mu$ such that
\[
 d_1(F_1)^{-1}\!\!\dots d_1(F_n)^{-1}(F_n\circ\dots\circ F_1)^*(\FS) \weakST \mu \;\; \text{as measures, as $n\to \infty$.}
\]
\end{result}

\noindent{In fact, the convergence
statement \eqref{E:asympFS} follows from the above result by taking $F_n = F$ for $n = 1, 2, 3,\dots$}

\begin{proof}[The proof of Theorem~\ref{T:d_1more}]
We assume that $\norm(F)\neq \varnothing$; there is nothing to prove otherwise.
Let us fix a $z_0\in \norm(F)$. Then, any $\Pth\in \path_n(z_0)$ is admissible, and
by Property\,(2) in sub-section~\ref{SS:normality}, assuming that $n\geq 2$, we get
\begin{align}
 \pi^0_{j-1}\circ\nu_{(\Pth,\,j-1)}\circ\psi_{(\Pth,\,j-1)} \ &= 
 \pi^0_{j-1}\circ\prjc_{j}\circ\nu_{(\Pth,\,j)}\circ\psi_{(\Pth,\,j)} \notag \\
 &= \pi^0_j\circ\nu_{(\Pth,\,j)}\circ\psi_{(\Pth,\,j)} \;\;\; \forall j=2,\dots,n. \notag
\end{align}
Iterating this argument, we deduce the following:
\begin{align}
 \text{For any $n\geq 2$}, \ \Pth &\in \path_n(z_0) \ \text{(and admissible)} \notag \\
 &\Rightarrow \ \pi^0_j\circ\nu_{(\Pth,\,j)}\circ\psi_{(\Pth,\,j)} \ = \
 \pi^0_1\circ\nu_{(\Pth,1)}\circ\psi_{(\Pth,1)} \;\;\; \forall j=1,\dots,n. \label{E:zeroRel}
\end{align}
Let us now {\em fix} a disc $\Delta$ around $0\in \cplx$ such that $\Delta\Subset \dom$ (where $\dom$ is as
given by Definition~\ref{D:NSet}).
Define: 
\[
 K \ := \ \bigcap_{\Opth\in \path_1(z_0)}\pi^0_1\circ\nu_{(\Opth,\,1)}\circ\psi_{(\Opth,\,1)}(\Dsc).
\]
Clearly, there is a region $G\Subset \dom$, containing $0$, such that
\[
 \left(\nu_{(\Opth,\,1)}\circ\psi_{(\Opth,\,1)}\right)^{-1}
 \Big(\Big(\left.\pi^0_1\right|_{U(\Opth,\,1)}\Big)^{\raisebox{-1pt}{$\scriptstyle {-1}$}}\!(K)\Big)\,\Subset\,G
 \; \; \; \text{for every $\Opth\in \path_1(z_0)$},
\]
where $U(\Opth,1)$ is as described in
sub-section~\ref{SS:normality}. From the above, from \eqref{E:zeroRel}, and from the fact that any
$\Pth\in \path_n(z_0)$ has some $\Opth\in \path_1(z_0)$ as an initial $1$-path, we deduce:
\begin{equation}\label{E:coverRel}
 \text{for any $n\geq 1$ \&} \ \Pth\in \path_n(z_0), \;
 \left(\nu_{(\Pth,\,n)}\circ\psi_{(\Pth,\,n)}\right)^{-1}
 \Big(\Big(\left.\pi^0_n\right|_{U(\Pth,\,n)}
 \Big)^{\raisebox{-1pt}{$\scriptstyle {-1}$}}\!(K)\Big)\,\Subset G. 
\end{equation}

We can deduce from Definition~\ref{D:NSet} that $K\subset \norm(F)$.
As $K$ has non-empty interior, it suffices to show that for any {\em non-negative} function
$\varphi\in \smoo(\pro;\rea)$ with ${\rm supp}(\varphi)\subset K$, 
$\int_{\pro}\varphi\,d\mu_F=0$. Hence, let us pick some function $\varphi\in \smoo(\pro;\rea)$ as
described. For any path $\Pth\in \path_n(z_0)$, let us write
\[
 \Ptmap \ := \ \left(\pi^0_n\times\pi^n_n\right)\circ\nu_{(\Pth,\,n)}\circ\psi_{(\Pth,\,n)}.
\]
We adopt notation analogous to that developed in the beginning of sub-section~\ref{SS:normality}.
For any multi-index $\mul{\alpha}\in \interR{1}{L}^j$, $\mul{\alpha}=(\alpha_1,\dots,\alpha_j)$,
$j=1,2,3,\dots$, let us define
\[
 {\IVar}^{\mul{\alpha}} \ := \ \begin{cases}
 							\IVar_{\alpha_1}, &\text{if $j=1$,} \\
 							\IVar_{\alpha_j}\circ {\IVar}^{(\alpha_1,\dots,\alpha_{j-1})}, &
 							\text{if $j\geq 2$.}
 							\end{cases}
 \]
If we write $\Pth$ as $(z_0,\dots,z_n;\,\mul{\alpha})$, then
$\Ptmap(\dom)$ is a ${\IVar}^{\mul{\alpha}}$-open neighbourhood of the point 
$(z_0, z_n)$. Furthermore, by our constructions in sub-section~\ref{SS:normality},
$\pi_1^{-1}(K)\cap |\corr^{\circ n}|$ is covered by the sets $\Ptmap(\dom)$ as
$\Pth$ varies through $\path_n(z_0)$. Hence, by definition:
\begin{equation}\label{E:1stCalc}
 \langle (F^n)^*(\FS), \varphi \rangle \ =
 \sum_{\Pth \in \path_n(z_0)}\;\int\limits_{{\rm reg}(\Ptmap(\dom))}
 \negthickspace\big(\left.\pi_1\right|_{\Ptmap(\dom)}\big)^*(\varphi)
 \big(\left.\pi_2\right|_{\Ptmap(\dom)}\big)^*(\FS).
\end{equation}
It is routine to show that
$\left.\pi^0_n\times\pi^n_n\right|_{U(\Pth,\,n)}$ is a branched covering map onto
its image. As $\nu_{(\Pth,\,n)}$ maps $\psi_{(\Pth,\,n)}(\dom)$ homeomorphically onto
$U(\Pth,n)$, the topological degree of
$\left.(\pi^0_n\times\pi^n_n)\circ\nu_{(\Pth,\,n)}\right|_{\raisebox{2pt}{$\scriptstyle \psi_{(\Pth,n)}(\dom)$}}$
equals the topological degree of $\left.\pi^0_n\times\pi^n_n\right|_{U(\Pth,\,n)}$.
Let us denote this number by ${\rm deg}_1(\Pth)$. By the change-of-variables formula, we
get:
\[
 \langle (F^n)^*(\FS), \varphi \rangle \ 
 =\sum_{\Pth \in \path_n(z_0)}\;\frac{1}{{\rm deg}_1(\Pth)}\int\limits_{\psi_{(\Pth,n)}(\dom)}
				\negthickspace(\varphi
					\circ\pi^0_n\circ\nu_{(\Pth,\,n)})
					(\pi^n_n\circ\nu_{(\Pth,\,n)})^*(\FS).
\]
Since $\Ptmap(\dom)$, in general, has singularities, we discuss briefly what is meant above by
``change-of-variables formula''. Note that:
\begin{itemize}
 \item The magnitude of the form $(\pi^n_n\circ\nu_{(\Pth,\,n)})^*(\FS)$ stays bounded
 on punctured neighbourhoods of any singular point of $\Ptmap(\dom)$.
 \item ${\rm reg}(\Ptmap(\dom))$ after at most finitely many punctures is the image of
 a Zariski-open subset of $\psi_{(\Pth,\,n)}(\dom)$ under a ${\rm deg}_1(\Pth)$-to-1 covering map.
\end{itemize}
Given these facts, it is a standard calculation that the right-hand side of \eqref{E:1stCalc} transforms to
the last integral above. For each $\Pth \in \path_n(z_0)$, let us write
\[
 {\rm deg}_2(\Pth) \ := \ \text{the degree of the map $\psi_{(\Pth,\,n)}:\dom\lrarw 
 \nu_{(\Pth,\,n)}^{-1}(U(\Pth,n))$}.
\]
By the change-of-variables formula for branched coverings of finite degree, we get:
\begin{align}
 \langle &(F^n)^*(\FS),\varphi\rangle \notag \\
 &=\negthickspace
 \sum_{\Pth \in \path_n(z_0)}\;\frac{1}{{\rm deg}_2(\Pth){\rm deg}_1(\Pth)}\int\limits_{\dom}
					\negthickspace(\varphi
                                        \circ\pi^0_n\circ\nu_{(\Pth,\,n)}\circ\psi_{(\Pth,\,n)})  
                                        (\pi^n_n\circ\nu_{(\Pth,\,n)}\circ\psi_{(\Pth,\,n)})^*(\FS) \notag \\
&=\negthickspace\sum_{\Pth \in \path_n(z_0)}\;\frac{1}{{\rm deg}_2(\Pth){\rm deg}_1(\Pth)}
					\int\limits_{\overline{G}}\negthickspace(\varphi
                                        \circ\pi^0_n\circ\nu_{(\Pth,\,n)}\circ\psi_{(\Pth,\,n)})
                                        (\pi^n_n\circ\nu_{(\Pth,\,n)}\circ\psi_{(\Pth,\,n)})^*(\FS). \label{E:cpt_suppSimpl}
\end{align}
The last expression follows from \eqref{E:coverRel} and the fact that ${\rm supp}(\varphi)\subset K$.
\smallskip

Endow $\pro$ with homogeneous coordinates. Given the form of the argument made below, we can assume without
loss of generality that $\pi^n_n\circ\nu_{(\Pth,\,n)}\circ\psi_{(\Pth,\,n)}(\dom)$ does not contain {\em both}
$[0:1]$ and $[1:0]$.
Write
\[
 \pi^n_n\circ\nu_{(\Pth,\,n)}\circ\psi_{(\Pth,\,n)} \ = \ [X_{(\Pth,\,n)}:Y_{(\Pth,\,n)}],
\]
where $X_{(\Pth,\,n)},Y_{(\Pth,\,n)}\in \hol(\dom)$ and have no common zeros in $\dom$, and define
\[
 q_{(\Pth,\,n)} \ := \ \begin{cases}
					X_{(\Pth,\,n)}/Y_{(\Pth,\,n)}, &
						\text{if $[0:1]\in \pi^n_n\circ\nu_{(\Pth,\,n)}\circ\psi_{(\Pth,\,n)}(\dom)$}, \\
					Y_{(\Pth,\,n)}/X_{(\Pth,\,n)}, &
						\text{if $[1:0]\in \pi^n_n\circ\nu_{(\Pth,\,n)}\circ\psi_{(\Pth,\,n)}(\dom)$}.
					\end{cases}
\]
From the expression for the Fubini--Study metric in local coordinates and from \eqref{E:cpt_suppSimpl},
we have the estimate (where $C$ is a constant\,$\geq 1$ that is independent of $n$):
\[
 d_1(F)^{-n}|\langle (F^n)^*(\FS),\varphi \rangle| \
 \leq \ \frac{C}{d_1(F)^n}\negthickspace\sum_{\Pth \in \path_n(z_0)}\;\int\limits_{\overline{G}}
 \|\varphi\|_{\infty}\bigg(\frac{|q_{(\Pth,\,n)}^\prime(\zt)|}{1+|q_{(\Pth,\,n)}(\zt)|^2}\bigg)^2dA(\zt).
\]
Since, by hypothesis, $\Fml(z_0)$ is a normal family, it follows by Marty's normality criterion --- see, for instance,
Conway \cite[Chapter VII/\S{3}]{conway:f1cv73} ---  that the family
\[
 \left\{\frac{|q_{(\Pth,\,n)}^\prime|}{1+|q_{(\Pth,\,n)}|^2}:
 n\in \zahl_+ \ \text{and $\Pth\in \path_n(z_0)$}\right\}
\]
is locally uniformly bounded. As $G\Subset \dom$, there exists an $M>0$ such that
\[
 \frac{|q_{(\Pth,\,n)}^\prime(\zt)|}{1+|q_{(\Pth,\,n)}(\zt)|^2} \ \leq \ M \;\; \forall \zt\in \overline{G}, \
 n\in \zahl_+ \ \text{and $\forall\Pth\in \path_n(z_0)$}.
\]
Given $n\in \zahl_+$, the number of summands in
\eqref{E:cpt_suppSimpl} is $d_{top}(\,{}^\dagger\!(F^n)) = d_0(F)^n$. The equality
is a consequence of Proposition~\ref{P:degreeSeq} and Lemma~\ref{L:iterate}. From it and the last two
estimates, it follows that
\begin{equation}\label{E:finalDecay}
 d_1(F)^{-n}|\langle (F^n)^*(\FS),\varphi \rangle| \ \leq \ C\left[\frac{d_0(F)}{d_1(F)}\right]^n\!M^2
 {\sf Area}(\overline{G}) \;\lrarw\; 0 \ \text{as $n\to \infty$,}
\end{equation}
since, by hypothesis, $d_0(F)<d_1(F)$.
\smallskip
 
In view of Result~\ref{R:1stRes}, taking $F_n = F$ for $n=1, 2, 3,\dots$, we have
\[
 \shortInt{\pro}{\varphi}\,d\mu_F \ = \ \lim_{n\to\infty}\frac{1}{d_1(F)^n}\langle (F^n)^*(\FS),\varphi \rangle
 \ = 0.
\]
In view of our remarks earlier, the theorem follows.
\end{proof}

\begin{remark}\label{Rem:horrors}
The reader will observe that to keep a count of multiplicities in the above proof, it is essential that the
integrands that appear be labelled by $\Pth$ in $\path_n(z_0)$ and by $n$ itself: i.e., the length of the forward
iteration of $F$ under consideration. If we do {\bf not} make the restriction $(*)$ about the paths to be considered
in Definition~\ref{D:NSet}, then it can happen that, for a path
$(z_0,\dots,z_N; \mul{\alpha})$, there exists some $j$, $1\leq j\leq N$, for which the
set $U(\Pth,j)$, as defined in Section~\ref{S:keyDefns} is {\em not} irreducible when
viewed as a {\em germ} of an analytic variety at $(z_0,\dots,z_j)$. Clearly, the label $(\Pth,j)$ wouldn't
then suffice to index the various analytic irreducible components into which $U(\Pth,j)$ 
splits. In short, the assumption of admissibility is made so that the basic motivation for the normality set is not
obscured by too much book-keeping paraphernalia.
\end{remark}
\medskip

\section{The proof of Theorem~\ref{T:d_1less} and Corollary~\ref{C:d1_less}}\label{S:proof-d_1less}

As the discussion in Section~\ref{S:results} preceding the statement of Theorem~\ref{T:d_1less} suggests,
its proof relies on several notions introduced in \cite{mcgehee:acrcHs92}. We therefore begin this section with
a definition and a couple of results from \cite{mcgehee:acrcHs92}.

\begin{definition}\label{D:attrBlock} 
Let $X$ be a compact Hausdorff space and let $f\subset X\times X$ be a relation of $X$ to itself such
that $\pi_1(f)=X$. A set $B\subset X$ is called an {\em attractor block} for $f$ if
$f(\overline{B}) \subset B^\circ$.
\end{definition}

\noindent{We recall that, given a relation $f$ and a set $S\subset X$, $f(S)$ is as defined in
Section~\ref{S:results}.}

\begin{result}[McGehee, Theorem 7.2 of \cite{mcgehee:acrcHs92}]\label{R:mcG_block}
Let $X$ be a compact Hausdorff space and let $f\subset X\times X$ be a relation of $X$
to itself such that $\pi_1(f)=X$. Assume $f$ is a closed set. If $B$ is an attractor block
for $f$, then $B$ is a neighbourhood of $\omega(B;f)$.
\end{result}

\begin{result}[McGehee, Theorem 7.3 of \cite{mcgehee:acrcHs92}]\label{R:mcG_chooseBlock}
Let $X$ be a compact Hausdorff space and let $f\subset X\times X$ be a relation of $X$
to itself such that $\pi_1(f)=X$. Assume $f$ is a closed set. If $\mathcal{A}$ is an attractor
for $f$ and $V$ is a neighbourhood of $\mathcal{A}$, then there exists a closed attractor
block $B$ for $f$ such that $B\subset V$ and $\omega(B;f)=\mathcal{A}$.
\end{result}

\noindent{We clarify that, given two subsets $A$ and $B$ of some topological space, $B$ is
called a neighbourhood of $A$ here (as in \cite{mcgehee:acrcHs92}) if $A\subset B^\circ$.}
\smallskip

Before we can give the proof of Theorem~\ref{T:d_1less}, we need one more concept. For this
purpose, we shall adapt some of the notations developed in Section~\ref{SS:normality}. Here,
$F$ will denote a holomorphic correspondence on $\pro$.
Firstly: given $N\in \zahl_+$, we say that $(w_0,w_{-1}\dots,w_{-N};\,\alpha_1,\dots,\alpha_N)\in
(\pro)^{N+1}\times\interR{1}{L}^N$ (see \eqref{E:alt} for the meaning of $L$) is {\em a
path of a backward iteration of $F$ starting at $w_0$, of length $N$}, if
\[ 
 (w_{-j},w_{-j+1})\in \IVar_{\alpha_j}, \; \; j=1,\dots,N.
\] 
In analogy with the notation in Section~\ref{SS:normality}, we set:
\begin{multline}
 \path_{-N}(w_0) \ := \ \text{the set of all paths of {\em backward} iterations} \\
							\text{of $F$, of length $N$, starting at $w_0$.} \notag
\end{multline}
Next, we say that a point $w\in \pro$ is a {\em regular value} of $F$ if it is not a critical value
(recall that we have defined this in Section~\ref{S:results}). We can now make the following
definition:

\begin{definition}\label{D:invRegBr}
Let $F$ be a holomorphic correspondence on $\pro$, let $N\in \zahl_+$, and let
$w_0\in \pro$. Let $\Ipth:=(w_0,w_{-1},\dots,w_{-N};\,\alpha_1,\dots,\alpha_N)\in
 \path_{-N}(w_0)$. We call the list $(\eff_{(\Ipth,\,1)},\dots,\eff_{(\Ipth,\,N)})$
a {\em regular branch of a backward iteration of $F$ along $\Ipth$} if:
\begin{itemize}
 \item[1)] $w_0,w_{-1},\dots,w_{-N+1}$ are regular values.
 \item[2)] For each $j=1,\dots,N$, $\eff_{(\Ipth,\,j)}$ is a holomorphic function defined by
 \[
  \eff_{(\Ipth,\,j)}(\zt) \ := \ \pi_1\circ\big(\left.\pi_2\right|_{\bran(\Ipth,\,j)}\big)^{-1}(\zt) \; \;
  \forall \zt\in \pi_2(\bran(\Ipth,j)),
 \]
 where $\bran(\Ipth,j)$ is a {\em local} irreducible component of $\IVar_{\alpha_j}$ at
 the point $(w_{-j},w_{-j+1})$ such that: $(a)$ $\bran(\Ipth,j)$ is smooth;
 $(b)$ $\pi_2$ restricted to $\bran(\Ipth,j)$ is
 injective; and $(c)$ $\pi_2(\bran(\Ipth,j))\supset \pi_1(\bran(\Ipth,j-1))$ when $j\geq 2$.
\end{itemize}   
\end{definition}

\noindent{The above is a paraphrasing --- for the setting in which we are interested --- of
the notion of a ``regular inverse branch of $F$ of order $N$'' introduced by Dinh in 
\cite{dinh:dpppcp05}.}
\smallskip

The following is the key proposition needed to prove Theorem~\ref{T:d_1less}.
\smallskip

\begin{proposition}\label{P:key-d_1less}
Let $F$ be a holomorphic correspondence of $\pro$ having all the properties stated
in Theorem~\ref{T:d_1less} and let $\rep$ be a strong repeller that is disjoint from
the set of critical values of $F$. Then, there exists a closed set $B\subset \pro$ such that
$B^\circ\supset \rep$ and such that:
\begin{itemize}
 \item[$i)$] The operator $\PerrF_B:= d_1(F)^{-1}\left.\Lam\right|_B$, where $\left.\Lam\right|_B$
 is as defined in \eqref{E:pullbackDirac2} with $B$ replacing $\Omega$, 
maps $\smoo(B;\cplx)$ into itself.
 \item[$ii)$] There exists a probability measure $\mu_B\in \smoo(\pro;\rea)^{\boldsymbol{*}}$
 that satisfies $\mu_B\circ\PerrF_B = \mu_B$ and such that
 \begin{equation}\label{E:aux_conv}
  \lim_{n\to \infty} \sup_B\left|\PerrF^n[\varphi]-\shortInt{B}{\varphi}\,d\mu_B\right| \ =
  \ 0 \quad\forall \varphi\in \smoo(B;\cplx).
 \end{equation}
\end{itemize}
\end{proposition}
\begin{proof}
Let $a_0\in \rep$ and let $U$ be an open set containing $\rep$ such that:
\begin{itemize}
 \item For each $w\in U$, there is a sequence $\{a_n(w)\}_{n\in \zahl_+}$ such that
 $a_n(w)\in \eff^n(w)$ for each $n$, and $a_n(w)\lrarw a_0$ as $n\to \infty$.
 \item $U$ contains no critical values of $F$.
\end{itemize}
By Result~\ref{R:mcG_block} and Result~\ref{R:mcG_chooseBlock}, we can find an
open neighbourhood $W$ of $\rep$ such that $\overline{W}\subset U$ and 
$\overline{W}$ is a closed attractor block for the relation $|\acorr|$.
\smallskip

Repeating the last argument once more, we can find a closed attractor block, $B$, for
$|\acorr|$ such that
\[
 \rep\,\subset\,B^\circ\,\subset\,B\,\subset\,W\,\Subset\,U.
\] 
By the above chain of inclusions and by the definition of the term ``attractor block'',
it follows that the operator $\PerrF_B$ maps $\smoo(B;\cplx)$ into itself.
\smallskip

\noindent{{\bf Claim 1.} {\em For each fixed $\varphi\in \smoo(B;\cplx)$, 
$\{\PerrF_B^n[\varphi]\}_{n\in \zahl_+}$ is an equicontinuous family.}}

\noindent{It is easy to see that $\rep$ is a closed {\em proper} subset. We can thus make a useful observation:
\begin{itemize}
 \item[$(**)$] We can choose $W$ so that
 $\pro\setminus \overline{W}$ is non-empty. Hence, we can choose coordinates in such a
 way that we may view $W$ as lying in $\cplx$, and that $W\Subset \cplx$. We shall work with 
 respect to these coordinates in the remainder of this proof.
\end{itemize}
Let us pick a point $w_0$ in $B$ (which, by construction, is a regular value) and let $D(w_0)$ be a
small disc centered at $w_0$ such that $\overline{D(w_0)}\subset W$. Let us fix an $N>1$ and consider
a path $\Ipth\in \path_{-N}(w_0)$. Recall that, by construction:
\begin{equation}\label{E:squeeze}
 \eff(\overline{W}) \ \subset \ W.
\end{equation}}

We can infer from \eqref{E:squeeze} that there exists a regular branch
$(\eff_{(\Ipth,\,1)},\dots,\eff_{(\Ipth,\,N)})$
of a backward iteration of $F$ along $\Ipth$. To see why, first note that, as $w_0$ is a regular value
and  $\overline{D(w_0)}\subset W$, we get:
\begin{enumerate}
 \item[$(a_1)$] There is an open neighbourhood $U_0$ of $w_0$ containing only regular values.
 \item[$(b_1)$] Writing 
 \[
  \widetilde{U}_1\,:=\,\text{the connected component of $\pi_2^{-1}(U_0)\cap\IVar_{\alpha_1}$
  containing $(w_{-1},w_0)$},
 \]
 and defining $\bran(\Ipth,1):=$\,any one of the irreducible components of the
 (local) complex-analytic variety $\widetilde{U}_1$, the function
 \[
  \eff_{(\Ipth,\,1)} \ := \ \pi_1\circ\big(\left.\pi_2\right|_{\bran(\Ipth,\,1)}\big)^{-1} \; \;
  \text{is holomorphic on $U_0$}.
  \]
 \item[$(c_1)$] The open set $U_1:=\eff_{(\Ipth,\,1)}(U_0)\subset W$ and hence contains only regular values.
\end{enumerate}
The assertion $(c_1)$ follows from the fact that $U_1\subset \eff(\overline{W})\subset W$
and that the latter contains no critical values.
\smallskip

Let us now, for some $k\in \zahl_+$,  $k\leq N-1$, {\em assume} the truth of the statements
$(a_k)$, $(b_k)$ and $(c_k)$, which are obtained by replacing all the subscripts $0$ and $1$ 
in $(a_1)$, $(b_1)$ and $(c_1)$ ({\em except} the subscript in $\pi_1$) by $k-1$ and $k$, respectively.
Now, $(a_{k+1})$ follows from $(c_k)$. Defining $\bran(\Ipth,k+1)$ in exact analogy to
$\bran(\Ipth,1)$, and writing
\[
 \eff_{(\Ipth,\,k+1)} \ := \ \pi_1\circ\big(\left.\pi_2\right|_{\bran(\Ipth,\,k+1)}\big)^{-1},
\]
the holomorphicity of $(\left.\pi_2\right|_{\bran(\Ipth,\,k+1)})^{-1}$ follows from 
$(a_{k+1})$. Thus $(b_{k+1})$ holds true.
We get $(c_{k+1})$ by appealing once again to \eqref{E:squeeze} and using the
fact that $U_{k}\subset W$. By induction, therefore,  a regular branch of a backward iteration of
$F$ along $\Ipth$ exists.
\smallskip

Furthermore, we can conclude that:
\[
 \eff_{(\Ipth,\,N)}\circ\dots\circ\eff_{(\Ipth,\,1)}(D(w_0))\subset W\Subset \cplx \ \ \
 \text{(see $(**)$ above)}.
\]
Recall that $\Ipth$ was arbitrarily chosen from $\path_{-N}(w_0)$ and that the arguments in
the last two paragraphs hold true for any choice of $\bran(\Ipth,j)$, $1\leq j\leq N$, and for any 
$N\in \zahl_+$. Thus, by Montel's theorem, we infer the following important fact: the family
\begin{align}
 \trp{\Fml}(w_0) \ := \ 
 &\left\{\eff_{(\Ipth,\,N)}\circ\dots\circ\eff_{(\Ipth,\,1)}\in \hol(D(w_0)): \Ipth\in \path_{-N}(w_0) \;
 \text{\em for some $N\in \zahl_+$} \; \right. \notag \\
 &\left.\text{\em \& $(\eff_{(\Ipth,\,N)},\dots,\eff_{(\Ipth,\,1)})$ is a regular branch of a backward
 iteration of $F$}\right\} \notag \\
 &\text{is a normal family.} \label{E:imp_normF}
\end{align}

Pick a $\varphi\in \smoo(B;\cplx)$ and let $\eps > 0$. As $B$ is compact, there exists a number
$\delta(\eps) > 0$ such that:
\begin{equation}\label{E:phiEst}
 |z_1-z_2| < \delta(\eps) \ \Rightarrow \ |\varphi(z_1)-\varphi(z_2)| < \eps \;\;\; \forall z_1, z_2\in B.
\end{equation}
We pick a $\zt\in B$. By taking $w_0 = \zt$ in the discussion in the previous paragraph, we infer
from the normality of the family $\trp{\Fml}(\zt)$ that we can find a sufficiently small number
$r(\eps,\zt) > 0$ such that:
\begin{equation}\label{E:innerEst}
 |\xi-\zt| < r(\eps,\zt) \ \text{and $\xi\in B$} \ \Rightarrow \ |\psi(\xi)-\psi(\zt)| < \delta(\eps) \;\;\;
 \forall\psi\in \trp{\Fml}(\zt).
\end{equation}
Now, for each $\zt\in B$  write:
\[
 \inB(N,\zt) \ := \ \{(\psi,\Ipth): \psi \ \text{is some $\eff_{(\Ipth,\,N)}\circ\dots\circ\eff_{(\Ipth,\,1)}$
 in $\trp{\Fml}(\zt)$}, \ \Ipth\in \path_{-N}(\zt)\}.
\]
If $(\zt_0,\zt_{-1},\dots,\zt_{-N};\mul{\alpha}) =: \Ipth$ is a path of backward iteration
(with $\zt_0$ being the above $\zt$), basic intersection theory tells us that the local intersection
multiplicity of $\IVar_{\alpha_j}$ with $\pro\times\{\zt_{-j+1}\}$ at  $(\zt_{-j},\zt_{-j+1})$ equals the
number of distinct branches $\eff_{(\Ipth,j)}$ one can construct according to the above inductive prescription
(this number is greater than $1$ if  $\IVar_{\alpha_j}$ has a normal-crossing singularity at
$(\zt_{-j},\zt_{-j+1})$). From this, and from the iterative construction of the $\eff_{(\Ipth,\,N)}$'s above,
it follows that:
\begin{equation}\label{E:ibCount}
 {\sf Card}[\,\inB(N,\zt)\,] \ = \ d_1(F)^N \;\;\; \forall \zt\in B.
\end{equation}
From \eqref{E:phiEst}, \eqref{E:innerEst} and \eqref{E:ibCount}, we get:
\begin{align}
 |\PerrF_B^n[\varphi](\xi)-\PerrF_B^n[\varphi](\zt)| \ = \
 &\left|\sum_{(\psi,\Ipth)\in \inB(n,\zt)}\!\!\!d_1(F)^{-n}
 (\varphi\circ\psi(\zt)-\varphi\circ\psi(\xi))\right| \notag \\
 \leq \ &d_1(F)^{-n}\!\!\!\sum_{(\psi,\Ipth)\in \inB(n,\zt)}\!\!\!|\varphi\circ\psi(\zt)-\varphi\circ\psi(\xi)| \
 < \ \eps \notag \\
 &\forall\xi \in B \ \text{such that $|\xi-\zt| < r(\eps,\xi)$ and $\forall n\in \zahl_+$.} \notag
\end{align}
The above holds true for each $\zt\in B$. This establishes Claim 1.
\smallskip

In what follows, the term {\em unitary spectrum} of an operator on a complex Banach space will mean
the set of all eigenvalues of the operator of modulus $1$, which we will denote by $\spU$. Observe
that $\spU(\PerrF_B)\ni 1$.
\smallskip

\noindent{{\bf Claim 2:} {\em $\spU(\PerrF_B)=\{1\}$, and the eigenspace associated with $1$ is
$\cplx$.}}

\noindent{The ingredients for proving the above claim
are largely those of \cite[\S{4}]{lyubich:epreRs83}. However,
to make clear the role that the properties of $\rep$ play, we shall rework some of the details of Lyubich's
argument. Let us fix a $\lambda\in \spU(\PerrF_B)$ and let $\varphi_\lambda\in \smoo(B;\cplx)$ be an
associated eigenfunction. Let $\zt_*\in B$ be such that 
$|\varphi_\lambda(\zt_*)| = \max_B|\varphi_\lambda|$. By definition
\begin{equation}\label{E:eigVect}
 d_1(F)^{-1}\!\!\!\sum_{(\psi,\Ipth)\in \inB(1,\zt_*)}\!\!\!\varphi_\lambda\circ\psi(\zt_*) \ 
 = \ \lambda\varphi_{\lambda}(\zt_*).
\end{equation}
Since $|\varphi_\lambda(\zt_*)| = \max_B|\varphi_\lambda|$, and given \eqref{E:ibCount},
the above equality would fail unless $|\varphi_\lambda\circ\psi(\zt_*)| = |\varphi_\lambda(\zt_*)|$
for each $\psi$ occurring in \eqref{E:eigVect}. It is now obvious from this fact and from 
\eqref{E:eigVect} that $\varphi_\lambda(x) = \lambda\varphi_\lambda(\zt_*) \ \forall
x\in \eff(\zt_*)$. Iterating, we get
\begin{equation}\label{E:repeat}
 \varphi_\lambda(x) \ = \  \lambda^n\varphi_\lambda(\zt_*) \quad \forall x\in \eff^n(\zt_*), \;\;
 n=1, 2, 3,\dots
\end{equation}
Since, by construction, $B\subset U$, there exists an $x_n\in \eff^n(\zt_*), \ n=1, 2, 3,\dots$, 
such that $x_n\lrarw a_0$. Therefore, owing to  \eqref{E:repeat}, the sequence
$\{\lambda^n\varphi_\lambda(\zt_*)\}_{n\in \zahl_+}$ is a convergent sequence. As
$\varphi_\lambda \not\equiv 0$ (by definition), this implies that $\lambda=1$.}
\smallskip

Observe that, having determined that $\lambda = 1$, \eqref{E:repeat} also gives
\begin{equation}\label{E:value}
 \varphi_\lambda(\zt_*) \ = \ \varphi_\lambda(a_0).
\end{equation}

Note that $\overline{\PerrF[\varphi]} =  \PerrF[\overline{\varphi}] \ \forall\varphi\in \smoo(B;\cplx)$.
Hence, $\er\varphi_\lambda$ and $\mi\varphi_\lambda$ are also eigenvectors of
$\PerrF$ associated to $\lambda = 1$. Thus, we have the following analogue of \eqref{E:eigVect}:
\[
 d_1(F)^{-1}\!\!\!\sum_{(\psi,\Ipth)\in \inB(1,\zz)}\!\!\!\er\varphi_\lambda\circ\psi(\zz) \ 
 = \ \lambda\er\varphi_{\lambda}(\zz),
\]
where $\zz\in B$ stands for either a point of global maximum or a point of global minimum of
$\er\varphi_\lambda$. Using the above as a starting point instead of \eqref{E:eigVect} and repeating,
with appropriate modifications, the argument that begins with \eqref{E:eigVect} and ends at
\eqref{E:value}, we get:
\begin{equation}\label{E:real-part}
 {\min}_B(\er\varphi_\lambda) \ = \ \er\varphi_\lambda(a_0) \ = \ {\max}_B(\er\varphi_\lambda).
\end{equation}
Similarly, we deduce that:
\[
 {\min}_B(\mi\varphi_\lambda) \ = \ \mi\varphi_\lambda(a_0) \ = \ {\max}_B(\mi\varphi_\lambda).
\]
Combining the above with \eqref{E:real-part}, we conclude that, for any eigenvector 
$\varphi_\lambda$ associated to $\lambda = 1$, $\varphi_\lambda\equiv constant$.
This establishes Claim 2.
\smallskip

To complete this proof, we need the following:

\begin{result}[Lyubich, \cite{lyubich:epreRs83}]\label{R:lyubich_Abs}
Let $\mathscr{B}$ be a complex Banach space. Let $A: \mathscr{B}\lrarw \mathscr{B}$ be a linear
operator such that $\{A^n(v)\}_{n\in \zahl_+}$ is a relatively-compact subset of $\mathscr{B}$
for each $v\in \mathscr{B}$. Assume that $\spU(A) = \{1\}$ and that $1$ is a simple eigenvalue.
Let $h\neq 0$ be an invariant vector of $A$. Then, there exists a linear functional $\mu$
that satisfies $\mu\circ A = \mu$ and $\mu(h)=1$, and such that
\[
 A^n(v) \lrarw \mu(v)h \;\; \text{as $n\to \infty$}
\]
for each $v\in \mathscr{B}$.
\end{result}  

Take $\mathscr{B} = \smoo(B;\cplx)$ in the above theorem. Note that 
$|\PerrF_B^n[\varphi]|$ is bounded by $\max_B|\varphi|$ for $n=1, 2, 3,\dots$
Thus, in view of Claims 1 and 2, $\PerrF_B$ satisfies all the hypotheses of Result~\ref{R:lyubich_Abs}.
Hence (recall that the function that is identically $1$ on $B$ is an eigenvector of $\PerrF_B$)
there is a regular complex Borel measure $\mu_B$ on $B$ such that $\int_B 1\,d\mu_B = 1$, and
\[
  \lim_{n\to \infty} \sup_B\left|\PerrF^n[\varphi]-\shortInt{B}{\varphi}\,d\mu_B\right| \ =
  \ 0 \quad\forall \varphi\in \smoo(B;\cplx).
\]
It is clear from the above equation that $\mu_B$ is a positive measure. Hence it is a probability
measure on $B$.
\end{proof}

\begin{proof}[The proof of Theorem~\ref{T:d_1less}]
Let $B$ be any closed
set having the properties listed in the conclusion of
Proposition~\ref{P:key-d_1less}. Let $\mu_B$ be the probability measure associated
to this $B$. We claim that $\mu_F$ is given by defining:
\[
 \shortInt{\pro}{\varphi}\,d\mu_F \ := \ \shortInt{B}{\left.\varphi\right|_B}\,d\mu_B \quad
 \forall \varphi\in \smoo(\pro;\cplx).
\]
We must show that $\mu_F$ does not depend on the choice of $B$. The proof of this
is exactly as given in \cite[Theorem\,1]{lyubich:epreRs83}. We fix a point $z\in \rep$.
So, $z\in B$ for any choice of $B$. Thus,
\begin{align}
 \shortInt{\pro}{\varphi}\,d\mu_F \ &= \ \lim_{n\to \infty}\PerrF_B^n[\varphi](z)
 && (\text{by Proposition~\ref{P:key-d_1less}}) \notag \\
 &= \ \lim_{n\to \infty}d_1(F)^{-n}\Lam^n[\varphi](z), 
 && (\text{since $\eff^n(z)\subset \eff^n(\rep)\subset \rep$}) \notag
\end{align} 
where $\Lam$ is as described in the passage following \eqref{E:pullbackDirac2}. The last line is independent
of $B$. Hence the claim.
\smallskip

By the above calculation, we also get $\int_{\pro}1\,d\mu_F = 1$. Thus, $\mu_F$ is
a probability measure.
\smallskip

Let $U$ be the open set described at the beginning of the proof of Proposition~\ref{P:key-d_1less}.
We now define:
\begin{multline}
 \opColl(\rep) \ := \ \left\{B\subset U: B \ \text{is closed}, \ B^\circ\supset \rep, \ \eff(B)\subset B \ 
 \text{and there exists a closed}\right. \\
  \left.\text{attractor block $B_*$ for $|\acorr|$ s.t.} \ B\subset B_*^\circ\subset B_*\subset U\right\}. \notag
\end{multline}
We see from the proof of Proposition~\ref{P:key-d_1less} that, owing to our hypotheses,
$\opColl$ is non-empty. Hence
\[
 U(F,\rep) \ := \bigcup\nolimits_{B\in \opColl(\rep)}B^\circ
\]
is a non-empty open set that contains $\rep$. Let $z\in U(F,\rep)$. There exists a $B\in \opColl$
such that $z\in B^\circ$. A close look at the essential features of its proof reveals that this $B$ 
has all the properties listed in the conclusion of Proposition~\ref{P:key-d_1less}. Consider any
$\varphi\in \smoo(\pro;\cplx)$. We now apply Lemma~\ref{L:iterate} to get
\begin{align}
 d_1(F)^{-n}\langle (F^n)^*(\delta_z), \varphi\rangle \ = \
 &d_1(F)^{-n}\Lam^n[\varphi](z) && (\text{from \eqref{E:pullbackDirac3} and Lemma~\ref{L:iterate}})
 \notag \\
 = \ &\PerrF_B^n[\varphi](z) && (\text{since $\eff^n(z)\subset \eff^n(B^\circ)\subset B$}) \notag \\
 &\lrarw \ \shortInt{\pro}{\varphi}\,d\mu_F \;\; \text{as $n\to \infty$}, \label{E:quasi-eqD}
\end{align}
and this holds true {\em for any $z\in U(F,\rep)$}. The last line follows from our observations above on
$\mu_F$. Now note that, by construction, $\eff(z)\subset U(F,\rep)$ for each $z\in U(F,\rep)$. Therefore,
in view of equation \eqref{E:pullbackDirac2}, it follows from \eqref{E:quasi-eqD} that 
$F^*(\mu_F) = d_1(F)\mu_F$.
\end{proof}

We are now in a position to provide:

\begin{proof}[The proof of Corollary~\ref{C:d1_less}]
Recall that, by definition, for any Borel measure $\mu$, $F_*(\mu):= (\eff)^*(\mu)$. Thus,
the proof of Corollary~\ref{C:d1_less} involves, in each case, applying one of the results above to $\eff$.
\smallskip

The proof $(i)$ follows from Fact~\ref{L:clarify} applied to $\eff$.
\smallskip

In view of Definitions~\ref{D:repeller} and \ref{D:strongRep} and the hypothesis of
part~$(ii)$, $\eff$ satisfies the conditions of Theorem~\ref{T:d_1less} (i.e., with $\eff$
replacing $F$). Thus, $(ii)$ follows from Theorem~\ref{T:d_1less}.
\end{proof} 

\section{Examples}\label{S:example}

The purpose of this section is to provide concrete examples that illustrate some of our comments
in Sections~\ref{S:intro} and \ref{S:results}. We begin by 
showing that it is easy to construct holomorphic correspondences
on $\pro$ that satisfy all the conditions stated in Theorem~\ref{T:d_1less}, but, unlike the examples
discussed in Remark~\ref{Rem:care}, 
have ``large'' repellers. Next, we shall discuss one of the
classes of holomorphic correspondences studied by Bullett and collaborators. For each correspondence
$F$ in this class, $d_0(F) = d_1(F) = 2$, and we shall show that Theorem~\ref{T:d_1less} and
Corollary~\ref{C:d1_less}$(ii)$ are applicable to these examples.
\vskip1mm

\subsection{Holomorphic correspondences addressed by Theorem~\ref{T:d_1less} having large repellers}\label{SS:large}
Choose a complex polynomial $p$ with ${\rm deg}(p)\geq 2$ such that its Julia set
$\mathcal{J}_p\varsubsetneq \pro$ and such that no critical values of $p$ lie in $\mathcal{J}_p$.
It follows --- see, for instance, \cite[\S{4}]{lyubich:epreRs83} --- that there is
a compact set $B$ such that $B^\circ\supset \mathcal{J}_p$ and avoids the critical values of $p$,
and such that  $p^{-1}(B)\subset B^\circ$. Next, choose a polynomial
$Q$ with ${\rm deg}(Q)\geq 2$ and having the following properties:
\begin{itemize}
 \item[$a)$] $Q$ has an attractive fixed point, call it  $\zt_0$, in $\mathcal{J}_p$; 
 \item[$b)$] $B$ lies in the basin of attraction (under the action of $Q$) of $\zt_0$.
\end{itemize}
In view of $(b)$, we can find a positive integer $N$ that is so large that
\begin{itemize}
 \item $Q^N(B) \subset B^\circ$;
 \item ${\rm deg}(Q)^N\geq {\rm deg}(p)$.
\end{itemize}
Let us write $q := Q^N$.
\smallskip

Next, we define:
\begin{align}
 \Gamma_1 \ &:= \ 
 \{([z_0:z_1], [w_0:w_1]): w_0^{{\rm deg}(q)}z_1-w_0^{{\rm deg}(q)}z_0\,q(w_1/w_0) = 0\}, \notag \\
 \Gamma_2 \ &:= \
 \{([z_0:z_1], [w_0:w_1]): w_1z_0^{{\rm deg}(p)}-w_0z_0^{{\rm deg}(p)}p(z_1/z_0) = 0\}. \notag
\end{align}
The projective coordinates are so taken that $[0:1]$ stands for the point at infinity.
\smallskip

We set $\corr := \Gamma_1+\Gamma_2$ and let $F$ denote the correspondence determined
by $\corr$. Clearly $d_0(F) = {\rm deg}(q)+1\geq {\rm deg}(p) +1 = d_1(F)$. In this type of construction,
we will have $d_0(F) > d_1(F)$ in general. However, apart from satisfying the rather coarse properties
$(a)$ and $(b)$ above, $p$ and $Q$ can be chosen with considerable independence from each other. Thus,
this construction will also produce holomorphic correspondences $F$ that satisfy $d_0(F) = d_1(F)$.
\smallskip

By construction, $\eff(B) \subset B^\circ$. In other words, $B$ is an attractor block for the relation $|\acorr|$.
Therefore, it follows from Result~\ref{R:mcG_block} and Definition~\ref{D:repeller} that $F$ has a repeller
$\rep\subset B^\circ$; this repeller is just $\omega(B; |\acorr|)$. 
\smallskip

It is easy to see that
\[
 \bigcap_{n\geq 0}\overline{\bigcup_{k\geq n}\eff^k(B)} \ \subseteq \ \omega(B;|\acorr|)
\]
(the reader may look up \cite[Theorem 5.1]{mcgehee:acrcHs92} for a proof). The above
implies that
\[
 \mathcal{J}_p\,\subseteq\,\rep\,\subseteq\;B^\circ.
\]
Thus, the correspondence $F$ defined above has a repeller that is disjoint from the set
of critical values of $F$. Owing to $(a)$ and $(b)$ above, for each $w\in B^\circ$, there
exists a point $a_n(w)\in \eff^n(w)$ such that $a_n(w)\lrarw \zt_0$. Hence, $\rep$ is
a strong repeller. Unlike the examples in Remark~\ref{Rem:care},
$\rep$ is ``large'' in a certain sense.
\vskip1mm

\subsection{On the mating between a quadratic map and a Kleinian group}\label{SS:Kleinian}
The ideas developed in the last section are of relevance to the correspondences --- alluded
to in Section~\ref{S:intro} --- introduced by Bullett and his
collaborators. We shall examine one such class of correspondences.
We shall not elaborate here upon what precisely is meant by the mating between a quadratic map on
$\pro$ and a Kleinian group. The idea underlying this concept is simple, but a precise definition
requires some exposition. We will just state here (rather loosely)
that such a mating provides a holomorphic
correspondence $F$ on $\pro$, and partitions $\pro$ into an open set and two-component closed set
(denoted below by $\Lamb$) --- both totally invariant under $F$ --- such that the action of the iterates of $F$
on $(\pro\setminus\Lamb)$ resembles the action of the given Kleinian group on its regular set, and the
iterates of distinguished branches of $F$ and $\eff$ --- on the components of $\Lamb$,
respectively --- resemble the dynamics of the given quadratic map on its filled Julia set.
We refer the reader to the introduction of \cite{bullettPenrose:mqmmg94}
or to \cite[\S{3}]{bullettHarvey:mqmKgqs00}.
The example we present here is that of a holomorphic correspondence on $\pro$ that realises
a mating between certain faithful discrete representations $r$ in $PSL_2(\cplx)$ of
\[
 G \ := \ \text{the {\em free} product of $\zahl_2$ and $\zahl_3$},
\]
and a quadratic map $q_c : z\longmapsto z^2 + c$. In this discussion, $r$ and $c$ will be such
that:
\begin{itemize}
 \item $q_c$ is hyperbolic and its filled Julia set, $\filJ(q_c)$, is homeomorphic to a closed disc;
 \item the regular set of $r$, $\OM(r)$, is connected.
\end{itemize}

The fact that is pertinent to this discussion is that the mating of the above two objects is
realisable as a holomorphic correspondence $F$ on $\pro$. This is the main result of the
article \cite{bullettHarvey:mqmKgqs00} by Bullett and Harvey. For such an $F$,
$d_0(F) = d_1(F) = 2$. (In fact, \cite[Theorem~1]{bullettHarvey:mqmKgqs00} establishes
the latter fact for a much larger class of maps $q_c$. For simplicity, however, we shall limit
ourselves to the assumptions above.)
\smallskip

We shall show that the above example satisfies all the conditions of Theorem~\ref{T:d_1less}
and briefly indicate how the hypothesis of Corollary~\ref{C:d1_less}$(ii)$ applies to it as well. 
\smallskip

Let $\corr$ denote the holomorphic $1$-chain that determines the correspondence provided by
\cite[Theorem~1]{bullettHarvey:mqmKgqs00}. Let us list some of the features of $F$
that are relevant to the present discourse (we will have to assume here that readers are familiar
with \cite{bullettHarvey:mqmKgqs00}): 
\begin{itemize}
 \item[1)] There exists a closed subset $\Lamb\subset \pro$ that is totally invariant under
 $F$ and is the disjoint union of two copies $\Lamb_+$ and $\Lamb_-$ of a homeomorph
 of a closed disc.
 \item[2)] There exist two open neighbourhoods $U^-$ and $V^-$ of $\Lamb_-$ such that
  \begin{itemize}
   \item[$a)$] $\overline{V^-}\subset U^-$;
   \item[$b)$] $\pi_1^{-1}(U^-)\cap |\corr|$ is the union of the graphs of two functions
   $f_j^+\in \hol(U^-)$, $j=1, 2$;
   \item[$c)$] $f_1^+(V^-) = U^-$;
   \item[$d)$] There is a quasiconformal homeomorphism of $U^-$ onto an open
   neighbourhood $\omega^-$ of $\filJ(q_c)$ that carries $\Lamb_-$ onto $\filJ(q_c)$,
   $\bdy\Lamb_-$ onto the Julia set of $q_c$, is conformal in the interior
   of $\Lamb_-$, and conjugates $f_1^+$ to $\left.q_c\right|_{\omega^-}$.
   \item[$e)$] $f_2^+(U^-)\cap U^- = \varnothing$.
  \end{itemize}
 \item[3)] There exist two open neighbourhoods $U^+$ and $V^+$, $\overline{V^+}\subset U^+$,
 of $\Lamb_+$ and a pair of functions $f_1^-, f_2^-\in \hol(U^+)$ such that
 \[
  \pi_2^{-1}(U^+)\cap |\corr| \ = \ \{(f_1^-(\zt),\zt) : \zt\in U^+\}\cup 
 									  \{(f_2^-(\zt),\zt) : \zt\in U^+\},
 \]
 and the analogues of the properties $2(c)$--$2(e)$ obtained by swapping the ``$+$'' and
 the ``$-$'' superscripts hold true.
 \item[4)] $(f_2^+)^{-1} = f_2^-$ and $\overline{U^-}\cap\overline{U^+}
   =\varnothing$.
\end{itemize}
We must record that property $(2)$ is stated in \cite[\S{3}]{bullettHarvey:mqmKgqs00}
without some of the features stated above. However, it is evident from Sections~3 and 4
of \cite{bullettHarvey:mqmKgqs00} that (under the assumptions stated at the beginning of
this subsection) there is a hybrid equivalence, in the sense of Douady--Hubbard, between 
$\left.f_1^+\right|_{\Lamb_-}$ and $\left.q_c\right|_{\filJ(q_c)}$.   
The Julia set of $q_c$, $\mathcal{J}(q_c)$, is a strong repeller for $q_c$ in the sense of
Definition~\ref{D:strongRep}. Furthermore (see \cite[\S{4}]{lyubich:epreRs83}, for instance) there
exists a neighbourhood basis of $\mathcal{J}(q_c)$, $\mathcal{N}_1 = \{W_\alpha\}$, say, such that
\begin{equation}\label{E:nbdb1}
 W_\alpha\,\subset\,\omega^+\cap\omega^-, \; \; \text{and} \; \;
 q_c^{-1}(\overline{W}_\alpha)\,\subset\,W_\alpha \quad \forall W_\alpha\in \mathcal{N}_1.
\end{equation}
Also note that, from properties $2(a)$--$2(e)$, the generic number of pre-images of any
$w\in U^-$ under $F$ equals $2$. From this and \eqref{E:nbdb1}, we can deduce that  
there exists a closed set $B$ such that $\bdy\Lamb_-\subset B^\circ \subset B$, 
$\eff(B)\subset B^\circ$ and
\[
 \bdy\Lamb_- \ = \ \bigcap_{n\geq 0} \eff^n(B).
\]
Since $\eff(B)\subset B^\circ$, it is not hard to show (see \cite[Theorem 5.4]{mcgehee:acrcHs92}, for
instance, for a proof) that the right-hand side above equals $\omega(B; |\acorr|\,)$. Then, by
definition, we have:   
\begin{itemize}
 \item[I)] $\bdy\Lamb_-$ is a repeller for $F$.
\end{itemize}
Furthermore, as $\mathcal{J}(q_c)$ is a strong repeller for the map $q_c$, invoking property $2(d)$
(along with the fact that the generic number of pre-images of any $w\in U^-$ under $F$ equals $2$)
gives us:
\begin{itemize}
 \item[II)] $F$ has the property described in Definition~\ref{D:strongRep} with $\rep = \Lamb_-$ and
 for some annular neighbourhood $U$ of $\bdy\Lamb_-$, whence $\bdy\Lamb_-$ is a strong repeller.
\end{itemize}

Now, let $Q: U^-\lrarw \omega^-$ be the quasiconformal homeomorphism such that 
$f_1^+ = Q^{-1}\circ q_c\circ Q$. Assume $\bdy\Lamb_-$ contains a critical value of $f_1^+$. By
the construction described above, and by property $2(d)$, $\bdy\Lamb_-$ is totally invariant under
$f_1^+$. Hence, there exists a point 
$\zt_0\in \bdy\Lamb_-$ such that $(f_1^+)^\prime(\zt_0) = 0$. There is a small connected
open neighbourhood $G$ of $\zt_0$ such that $G\subset U^-$ and $G\setminus \{\zt_0\}$
contains two distinct pre-images under $f_1^+$ of each point in $U^-$
belonging to a sufficiently small deleted neighbourhood of $f_1^+(\zt_0)$.
As $q_c = Q\circ f_1^+\circ Q^{-1}$ and $Q$ is a homeomorphism, an analogous statement
holds true around $Q(\zt_0)\in \mathcal{J}(q_c)$. By the inverse function theorem, this is
impossible because $\mathcal{J}(q_c)$ contains no critical points of $q_c$. The last statement is
a consequence of hyperbolicity --- see, for instance \cite[Theorem~3.13]{mcmullen:cdr94}. Hence, our assumption
above must be false. Finally, by property $2(e)$, any critical values of $f_2^+$ lie away from
$\Lamb_-$. Thus,
\begin{itemize}
 \item[III)] $F$ has no critical values on $\bdy\Lamb_-$.
\end{itemize}
We see that the properties (I)--(III) above are precisely the conditions 
of Theorem~\ref{T:d_1less} with $\rep = \bdy\Lamb_-$.
\smallskip

We conclude our discussion of the present example by observing that, in view of property $(3)$ above
and an argument analogous to the one that begins with the equation \eqref{E:nbdb1}, with $F$ replacing
$\eff$ and $\Lamb_+$ replacing $\Lamb_-$, we can also show that
Corollary~\ref{C:d1_less}$(ii)$ applies to this example.

\end{document}